\RequirePackage{latexsym, amsmath, amstext, amssymb, amsfonts, amscd, bm, array, amsbsy, mathrsfs, amscd}
\documentclass[11pt, cmp, numbook, final, a4paper]{svjour}
\usepackage[11pt]{extsizes}
\pdfoutput=1
\usepackage{indentfirst}
\usepackage{graphics} 
\usepackage{epsfig}
\usepackage{transparent, colortbl}
\usepackage{booktabs}
\usepackage{multirow}
\usepackage{anysize}
\usepackage{latexsym}
\usepackage{pdfsync}
\usepackage[all]{xy}
\usepackage[usenames,dvipsnames]{xcolor}
\usepackage{enumerate}
\usepackage{float}
\restylefloat{table}
\usepackage{upgreek}
\usepackage{array}
\def\rOmega{\mathrm{\Omega}}
\def\rGamma{\mathrm{\Gamma}}
\def\rH{\mathrm{H}}
\def\hK{{\widehat K}}
\def\hd{{\hat \dd}}
\def\hdelta{{\hat \updelta}}
\def\bF{\mathbf{F}}
\def\rPi{\mathrm{\Pi}}

\usepackage[colorlinks=true,linkcolor=black]{hyperref}
\colorlet{linkequation}{blue}

\newcommand*{\SavedEqref}{}
\let\SavedEqref\eqref
\renewcommand*{\eqref}[1]{%
  \begingroup
    \hypersetup{
      linkcolor=blue,
      linkbordercolor=blue,
    }%
    \SavedEqref{#1}%
  \endgroup
}

\makeatletter
\def\CT@@do@color{%
  \global\let\CT@do@color\relax
        \@tempdima\wd\z@
        \advance\@tempdima\@tempdimb
        \advance\@tempdima\@tempdimc
        \kern-\@tempdimb
\transparent{0.6}%
        \leaders\vrule
                \hskip\@tempdima\@plus  1fill
        \kern-\@tempdimc
        \hskip-\wd\z@ \@plus -1fill }
\makeatother

\makeatletter
\newcommand{\thickhline}{%
    \noalign {\ifnum 0=`}\fi \hrule height 1pt
    \futurelet \reserved@a \@xhline
}
\newcolumntype{"}{@{\hskip\tabcolsep\vrule width 1pt\hskip\tabcolsep}}
\makeatother

\newtheorem{Theorem}{Theorem}[section]
\newtheorem{Lemma}[Theorem]{Lemma}
\newtheorem{Proposition}[Theorem]{Proposition}
\newtheorem{Corollary}[Theorem]{Corollary}

\newtheorem{Definition}[Theorem]{Definition}

\newcommand{\Hom}{{\rm Hom}}

\newcommand{\im}{\mathrm{im}}

\let\SSS\S
\let\oldqed\qed
\renewcommand\qed{\nobreak\enspace\oldqed}

\newcommand{\Ext}{\mathrm{Ext}}

\newcommand{\bp}{\begin{Proposition}}
\newcommand{\ep}{\end{Proposition}}
\newcommand{\bl}{\begin{Lemma}}
\newcommand{\el}{\end{Lemma}}
\newcommand{\bt}{\begin{Theorem}}
\newcommand{\et}{\end{Theorem}}
\newcommand{\End}{\mathrm{End}}
\newcommand{\Aut}{\mathrm{Aut}}

\newcommand{\Mod}{\mathrm{Mod}}

\newcommand{\ev}{\mathrm{ev}}
\newcommand{\eqdef}{\stackrel{{\rm def.}}{=}}

\newcommand{\cinf}{{C^\infty(X)}}

\DeclareFontFamily{U}{rsf}{}
\DeclareFontShape{U}{rsf}{m}{n}{<5> <6> rsfs5 <7> <8> <9> rsfs7 <10-> rsfs10}{}
\DeclareMathAlphabet\Scr{U}{rsf}{m}{n}

\def\cE{\check{E}}

\def\bE{\mathbf{E}}
\def\bhE{\widehat{\mathbf{E}}}

\def\Z{\mathbb{Z}}
\def\C{\mathbb{C}}

\def\H{\mathbb{H}}

\def\deg{{\rm deg\,}}

\def\dd{\mathrm{d}}

\def\fd{\mathfrak{d}}
\def\md{\boldsymbol{\fd}}

\def\rD{\mathrm{D}}

\def\i{\mathbf{i}}
\def\cC{\mathcal{C}}

\def\bhd{\hat{\mathbf{d}}}
\def\uHom{\underline{\Hom}}
\def\Ev{\mathrm{Ev}}
\def\sm{\mathrm{sm}}

\newcommand{\be}{\begin{equation*}}
\newcommand{\ee}{\end{equation*}}
\newcommand{\ben}{\begin{equation}}
\newcommand{\een}{\end{equation}}
\newcommand{\beqa}{\begin{eqnarray*}}
\newcommand{\eeqa}{\end{eqnarray*}}
\newcommand{\beqan}{\begin{eqnarray}}
\newcommand{\eeqan}{\end{eqnarray}}

\newcommand{\CC}{\ensuremath{\mathbb{C}}}

\newcommand \pd {{\partial}}
\newcommand \bpd {{\overline{\partial}}}

\newcommand{\nn}{\nonumber}

\newcommand{\id}{\mathrm{id}}
\newcommand{\Tr}{\mathrm{Tr}}
\newcommand{\tr}{\mathrm{tr}}
\newcommand{\str}{\mathrm{str}}

\def\Gr{\mathrm{Gr}}
\def\cK{\mathrm{\cal K}}

\def\O{\mathrm{O}}

\def\cA{\mathcal{A}}
\def\cE{\mathcal{E}}

\def\cT{\mathcal{T}}

\def\G_2{\mathrm{G_2}}
\def\cO{\mathcal{O}}

\def\cS{\mathcal{S}}
\def\S{\mathbb{S}}

\def\s{\mathfrak{s}}

\def\Ob{\mathrm{Ob}}
\def\VB{\mathrm{VB}}

\def\sc{\mathrm{sc}}

\def\PV{\mathrm{PV}}
\def\HPV{\mathrm{HPV}}

\def\ioda{\boldsymbol{\iota}}
\def\bbpd{\boldsymbol{\bpd}}
\def\hbbpd{\widehat{\bbpd}}

\def\0{{\hat{0}}}
\def\1{{\hat{1}}}

\def\O{\mathrm{O}}
\def\DF{\mathrm{DF}}
\def\HDF{\mathrm{HDF}}

\def\vect{\mathrm{vect}}

\def\hPV{\widehat{\PV}}
\def\hupdelta{\hat{\updelta}}
\def\hHPV{\widehat{\HPV}\phantom{}}
\def\hF{{\widehat F}}
\def\hbF{{\widehat \bF}}
\def\hOmega{{\widehat \rOmega}}
\def\hcA{{\widehat \cA}}
\def\hHDF{{\widehat \HDF}}
\def\hsigma{{\widetilde \sigma}}

\newcommand{\gr}{\mathrm{gr}\,}

\def\bd{\mathbf{d}}

\def\sq{\mathrm{sq}}

\newcommand{\twopartdef}[4]
{
	\left\{
		\begin{array}{ll}
			#1 & \mbox{if } #2 \\
			#3 & \mbox{if } #4
		\end{array}
	\right.
}

\newtheorem{innercustomthm}{Theorem}
\newenvironment{customthm}[1]
  {\renewcommand\theinnercustomthm{#1}\innercustomthm}
  {\endinnercustomthm}

\begin{document}

\title{Non-degeneracy of cohomological traces for general Landau-Ginzburg models}

\author{Dmitry Doryn \and Calin Iuliu Lazaroiu}

\institute{Center for Geometry and Physics, Institute for Basic
  Science (IBS), Pohang 37673, Republic of Korea\\
\email{dmitry@ibs.re.kr, calin@ibs.re.kr}}

\date{}

\maketitle

\abstract{We prove non-degeneracy of the cohomological bulk and
boundary traces for general open-closed Landau-Ginzburg models
associated to a pair $(X,W)$, where $X$ is a non-compact complex
manifold with trivial canonical line bundle and $W$ is a
complex-valued holomorphic function defined on $X$, assuming only that
the critical locus of $W$ is compact (but may not consist of isolated
points). These results can be viewed as certain ``deformed'' versions
of Serre duality. The first amounts to a duality property for
the hypercohomology of the sheaf Koszul complex of $W$, while the second is
equivalent with the statement that a certain power of the shift 
functor is a Serre functor on the even subcategory of
the $\Z_2$-graded category of topological D-branes of such models. }

\tableofcontents 

\section*{Introduction}
\label{sec:intro}

In references \cite{nlg1,nlg2}, we discussed a general framework for
Landau-Ginzburg models inspired by the physics arguments of
\cite{LG1,LG2} and conjectured that certain objects associated to a
K\"ahlerian Landau-Ginzburg pair obey the defining axioms of an
open-closed TFT datum (see \cite{tft}). In the present paper, we
establish non-degeneracy of the cohomological bulk and boundary traces
constructed in \cite{nlg1}, thus proving part of that conjecture.
Since the bulk and boundary traces can be defined without the
K\"ahlerianness requirement, our proof does not assume that condition.

By definition, a {\em holomorphic Landau-Ginzburg pair} (LG pair) is a doublet
$(X,W)$ such that\footnote{A holomorphic LG pair is called {\em K\"ahlerian} 
if $X$ admits at least one K\"ahler metric.}:
\begin{enumerate}
\item $X$ is a non-compact complex manifold (which we assume to be
  connected and paracompact) with holomorphically trivial canonical line bundle.
\item $W:X\rightarrow \C$ is a non-constant holomorphic function.
\end{enumerate}

\

\noindent Let $(X,W)$ be a holomorphic Landau-Ginzburg pair such that
$\dim_\C X=d$. Let:
\ben
\mu\eqdef \hat d \in \Z_2
\een
be the modulo 2 reduction of $d$, which is known as the {\em signature} 
of the LG pair $(X,W)$. Let:
\ben
\label{ZWdef}
Z_W\eqdef \big\{x\in X\,\big|\,(\pd W)(x)=0\big\}
\een
denote the critical set of $W$. Let $\O(X)$ be the commutative ring of
complex-valued holomorphic functions defined on $X$. 

The {\em cohomological bulk algebra} of $(X,W)$ (see \cite[Section 3]{nlg1}) is
the $\Z$-graded associative and supercommutative algebra $\HPV(X,W)$
defined as the total cohomology algebra of the $\O(X)$-linear
dg-algebra $(\PV(X),\updelta_W)$, where the $\O(X)$-module:
\be
\PV(X)\eqdef \bigoplus_{i=0}^d \bigoplus_{j=0}^d \cA^i(X,\wedge^j TX)
\ee
is endowed with the multiplication given by the wedge product and with
the $\Z$-grading which places $\cA^i(X,\wedge^j T X)\eqdef
\rOmega^{0,i}(X,\wedge ^j TX)$ in degree $i-j$. The {\em
twisted differential\,} $\updelta_W$ is defined through:
\ben
\label{updelta}
\updelta_W\eqdef \bpd_{\wedge TX}+\ioda_W~~,
\een
where $\bpd_{\wedge TX}$ is the Dolbeault differential of the
holomorphic vector bundle $\wedge TX\eqdef \oplus_{j=0}^d \wedge^j TX$
and $\ioda_W\eqdef -\i \pd W\lrcorner $, where $\i$ denotes the
imaginary unit.

Suppose that the critical set of $W$ is compact. In this case, it was
shown in \cite{nlg1} that $\HPV(X,W)$ is finite-dimensional over $\C$
and that any holomorphic volume form $\Omega$ on $X$ induces a natural
$\C$-linear map $\Tr_\Omega:\HPV(X,W)\rightarrow \C$ of degree $0$
(known as the {\em bulk cohomological trace}, see \cite[Section
5]{nlg1}) which is graded cyclic in the sense that it satisfies:
\be
\Tr_\Omega(u_1u_2)=(-1)^{\deg u_1\deg u_2}\Tr_\Omega(u_2u_1)
\ee
for any elements $u_1,u_2\in \HPV(X,W)$ which are
homogeneous with respect to the grading on $\HPV(X,W)$. In this paper, we prove:

\

\begin{customthm}{A.}
\label{thm:A}
Suppose that the critical set $Z_W$ is compact. Then the
graded-symmetric pairing
$\langle\cdot,\cdot\rangle_\Omega:\HPV(X,W)\times\HPV(X,W)\rightarrow
\C$ given by:
\be
\langle u_1,u_2\rangle_\Omega\eqdef\Tr_\Omega(u_1 u_2)~,~~\forall u_1,u_2\in \HPV(X,W)~~
\ee
is non-degenerate. Hence $(\HPV(X,W),\Tr_\Omega)$ is a finite-dimensional\,
$\Z$-graded supercommutative Frobenius algebra over $\C$.
\end{customthm}

\

\noindent Let $\cK_W$ denote the sheaf Koszul complex defined by $\ioda_W$:
\ben
\label{cKdef}
(\cK_W):~~~~~~~~~~~0\longrightarrow \wedge^{d}T X\stackrel{\ioda_W}{\longrightarrow} \wedge^{d-1} TX 
\stackrel{\ioda_W}{\longrightarrow} \ldots \stackrel{\ioda_W}{\longrightarrow} TX 
\stackrel{\ioda_W}{\longrightarrow} \cO_X\longrightarrow 0~~.
\een

\noindent When $Z_W$ is compact, we have isomorphisms $\HPV^k(X,W)\simeq \H^k(\cK_W)$, where 
$\H(\cK_W)$ is the hypercohomology of the bounded complex of analytic sheaves $\cK_W$.
Thus Theorem \ref{thm:A} has the equivalent formulation: 

\

\begin{customthm}{A$'$.}
\label{thm:A'}
Suppose that the critical set $Z_W$ is compact. Then for every $k\in \{-d,\ldots, d\}$ we
have isomorphisms of vector spaces: 
\be 
\H^k(\cK_W)\simeq_\C \H^{-k}(\cK_W)^\vee~~, 
\ee
which depend on the choice of a holomorphic volume form $\Omega$.  
\end{customthm}

\

Another fundamental datum defined by the holomorphic LG pair $(X,W)$ is the {\em
  cohomological twisted Dolbeault category of holomorphic
  factorizations} $\HDF(X,W)$ (see \cite[Section 4]{nlg1}). This is a
$\Z_2$-graded $\O(X)$-linear category defined as the total cohomology
category of a $\Z_2$-graded $\O(X)$-linear dg-category $\DF(X,W)$
whose objects are the holomorphic factorizations of $W$. By
definition, a {\em holomorphic factorization} of $W$ is a pair
$(E,D)$, where $E=E^\0\oplus E^\1$ is a $\Z_2$-graded holomorphic
vector bundle defined on $X$ and $D$ is a holomorphic section of the
bundle $End^\1(E)\eqdef Hom(E^\0,E^\1)\oplus Hom(E^\1,E^\0)$ which
satisfies the condition $D^2=W \id_E$. Given two holomorphic
factorizations $a_1=(E_1,D_1)$ and $a_2=(E_2,D_2)$ of $W$, the space
of morphisms from $a_1$ to $a_2$ in the category $\DF(X,W)$ is the
$\O(X)$-module:
\ben
\label{HomDF}
\Hom_{\DF(X,W)}(a_1,a_2)\eqdef \cA(X,Hom(E_1,E_2))=\bigoplus_{i=0}^d\cA^i (X,Hom(E_1,E_2))~~,
\een
endowed with the $\Z_2$-grading which places the submodule
$\oplus_{{\hat i}+\kappa=\tau}\cA^i (X,Hom^\kappa(E_1,E_2))$ (where
$i\in \{1,\ldots, d\}$ and $\kappa\in \Z_2$) in degree $\tau\in
\Z_2$. This module is endowed with the differential:
\ben
\label{tdolbeault}
\updelta_{a_1,a_2}\eqdef \bpd_{Hom(E_1,E_2)} + \fd_{a_1,a_2}~~, 
\een
where $\bpd_{Hom(E_1,E_2)}$ is the Dolbeault differential
of the holomorphic vector bundle $Hom(E_1,E_2)$ and $\fd_{a_1,a_2}$
is the {\em defect differential}, which is uniquely determined by the
condition:
\ben
\label{defect}
\fd_{a_1,a_2}(\omega\otimes f)=(-1)^i \omega \otimes (D_2\circ f) - (-1)^i (-1)^\kappa \omega \otimes (f\circ D_1)~~,
\een 
for all $\omega\in \cA^i(X)$ and $f\in
\rGamma(X,Hom^\kappa(E_1,E_2))$. The composition of morphisms in
$\DF(X,W)$ is induced in the obvious manner by the wedge product and
by the fiberwise composition of linear maps.

Suppose that the critical locus of $W$ is compact. In this case, it
was shown in \cite{nlg1} that $\HDF(X,W)$ is Hom-finite as a
$\C$-linear category and that any holomorphic volume form $\Omega$ on
$X$ naturally induces a $\C$-linear map
$\tr^\Omega_a:\Hom_{\HDF(X,W)}(a,a)\rightarrow \C$ of $\Z_2$-degree
${\hat d}$ (called the {\em cohomological boundary trace}
\cite[Section 6]{nlg1}) for any holomorphic factorization $a$ of $W$,
such that the following graded cyclicity condition is satisfied for
any two holomorphic factorizations $a_1,a_2$ of $W$:
\be
\tr_{a_1}^\Omega(t_2 t_1)=(-1)^{\kappa_1\kappa_2} \tr_{a_2}^\Omega (t_1t_2)~,~
~\forall t_1\in \Hom_{\HDF(X,W)}^{\kappa_1}(a_1,a_2)~,~~\forall t_2\in \Hom_{\HDF(X,W)}^{\kappa_2}(a_2,a_1)~~,
\ee
where $\kappa_1,\kappa_2\in \Z_2$. Let $\tr^\Omega$ denote the family
$(\tr_a^\Omega)_{a\in \Ob\HDF(X,W)}$. The second main result of this paper
 is:

\

\begin{customthm}{B.}
\label{thm:B}
Suppose that the critical set $Z_W$ is compact. Then the bilinear
pairing
\be
\langle\cdot,\cdot\rangle_{a_1,a_2}^\Omega:\Hom_{\HDF(X,W)}(a_1,a_2)\times
\Hom_{\HDF(X,W)}(a_2,a_1)\rightarrow \C
\ee
defined through:
\be
\langle t_1,t_2\rangle_{a_1,a_2}^\Omega\eqdef \tr_{a_2}^\Omega(t_1 t_2)~~,
~\forall t_1\in \Hom_{\HDF(X,W)}(a_1,a_2)~~,~\forall t_2\in \Hom_{\HDF(X,W)}(a_2,a_1)~
\ee 
is non-degenerate for any two holomorphic factorizations $a_1,a_2$ of
$W$. Hence $(\HDF(X,W),\tr^\Omega)$ is a Calabi-Yau supercategory
of signature $\mu=\hat d$ in the sense of {\rm \cite[Section 2]{nlg1}}.
\end{customthm}

\

\noindent The $\Z_2$-graded category $\HDF(X,W)$ admits an automorphism $\Sigma$
which squares to the identity and comes with natural isomorphisms: 
\be
\Hom_{\HDF(X,W)}(a_1,\Sigma(a_2))\simeq \Hom_{\HDF(X,W)}(\Sigma(a_1), a_2)\simeq \rPi\Hom_{\HDF(X,W)}(a_1, a_2)~~,
\ee
where $\rPi$ is the parity change functor of the category of
$\Z_2$-graded vector spaces. As a consequence, $\HDF(X,W)$ can be
reconstructed from its even subcategory $\HDF^\0(X,W)$ (which is
obtained from $\HDF(X,W)$ by keeping only morphisms of $\Z_2$-degree equal to $\0$). 
Then Theorem \ref{thm:B} can be reformulated as follows:

\

\begin{customthm}{B$'$.}
\label{thm:B'}
Suppose that the critical set $Z_W\!$ is compact. Then $\Sigma^d\!$ is a Serre functor
for the category $\HDF^\0\!(\!X,\!W\!)$, where:
\be
\Sigma^d=\twopartdef{\Sigma~}{d~\mathrm{is~odd}}{\id_{\HDF^\0(X,W)}~~}{d~\mathrm{is~even}}~~.
\ee
\end{customthm}

\

\noindent Notice that $\Sigma^d$ depends only on the signature $\mu=\hat d$.

\

The differentials $\updelta_W=\bpd_{\wedge
  TX}+\ioda_W$ and
$\updelta_{a_1,a_2}=\bpd_{Hom(E_1,E_2)}+\fd_{a_1,a_2}$ can be viewed
as deformations of the Dolbeault operators $\bpd_{\wedge TX}$ and
$\bpd_{Hom(E_1,E_2)}$ respectively.  When $W=0$, the differential
$\updelta_W$ reduces to $\bpd_{\wedge TX}$ and one can check that
Theorem \ref{thm:A} reduces to Serre duality (on the non-compact
complex manifold $X$) for the holomorphic vector bundle $\wedge
TX$. In this case, a particular class of holomorphic factorizations is
given by pairs of the form $(E,D)=(E,0)$, for which Theorem
\ref{thm:B} again reduces to Serre duality. Accordingly, both
theorems can be viewed as ``deformed'' versions of ordinary Serre
duality \cite{Serre,RR,AK} on (non-compact) complex manifolds. We will
prove them by reduction to the latter by using certain spectral
sequences which relate the cohomology of the differentials
$\updelta_W$ (respectively $\updelta_{a_1,a_2}$) to the Dolbeault
cohomology of the holomorphic vector bundles ${\wedge TX}$
(respectively $Hom(E_1,E_2)$).

The paper is organized as follows. Section \ref{sec:duality} recalls
some well-known facts regarding duality for complexes of
Fr\'echet-Schwartz (FS) and dual of Fr\'echet-Schwartz (DFS) spaces.
Section \ref{sec:spec} proves some results regarding spectral
sequences which will be used later on. In Section \ref{sec:graded},
we discuss an extension of the notion of Serre pairing to the case of
graded holomorphic vector bundles. Section \ref{sec:bulk} gives our
proof of Theorems A and A$'$, while Section \ref{sec:boundary} proves
Theorems B and B$'$. Appendix \ref{app:shift} collects some properties of linear
categories and supercategories with involutive shift functor.

\paragraph{Notations and conventions}

We use the same notations and conventions as in reference \cite{nlg1}. 
In particular, the symbol $\Z_2$ stands for the field $\Z/2\Z$, 
whose elements we denote by $\0$ and $\1$. Given an integer $n\in \Z$, 
we denote its reduction modulo $2$ by ${\hat n}\in \Z_2$. The symbol $\i$ 
denotes the imaginary unit, while $\lrcorner$ denotes the 
contraction between differential forms and polyvector fields. 

Throughout the paper, $X$ is a {\em connected} and non-compact complex
manifold with holomorphically trivial canonical line bundle and $W$ is a non-constant
holomorphic complex-valued function defined on $X$. Since $X$ is
connected, the following conditions are equivalent:
\begin{enumerate}[(a)] \itemsep 0.0em
\item $X$ is paracompact
\item $X$ is second countable
\item $X$ is $\sigma$-compact
\item $X$ is countable at infinity.
\end{enumerate} 
We assume throughout the paper that one (and hence
all) of these equivalent conditions is satisfied by $X$.  The
$\C$-algebra of smooth complex-valued functions defined on $X$ is
denoted by $\cC^\infty(X)$, while the $\C$-algebra of holomorphic
complex valued functions defined on $X$ is denoted by $\O(X)$; the
sheaf of locally-defined holomorphic complex-valued functions is
denoted by $\cO_X$. Given a holomorphic vector bundle $V$ defined on
$X$, its $\O(X)$-module of globally-defined holomorphic sections is
denoted by $\rGamma(X,V)$ while its $\cC^\infty(X)$-module of
globally-defined smooth sections is denoted by $\rGamma_\sm(X,V)$. We
sometimes tacitly identify a holomorphic vector bundle with its sheaf
of local holomorphic sections.

\section{Duality for complexes of topological vector spaces}
\label{sec:duality}

In this section, we summarize some properties of complexes of Fr\'echet-Schwartz (FS) and 
duals of Fr\'echet-Schwartz (DFS) spaces, following \cite{Grot,MV,Cassa,TL}. 

Throughout this paper, a {\em topological vector space}
(tvs) means a topological vector space over the normed field $\C$ of
complex numbers. Given a tvs $F$, let $F^\ast$ denote the topological
dual of $F$, endowed with the strong topology. Given a continuous
linear map $f:F_1\rightarrow F_2$ between two topological vector
spaces, we denote its transpose by $f^t:F_2^\ast\rightarrow
F_1^\ast$; the transpose of $f$ is continuous with respect to the strong
topologies on the dual spaces. A continuous linear map $f:F_1\rightarrow F_2$ is called a
{\em topological homomorphism} if the corestriction
$f_0:F_1\rightarrow f(F_1)$ of $f$ to its image is an open map when
the image subspace $f(F_1)$ is endowed with the induced topology. The
linear map $f$ is called a {\em topological isomorphism} if it is a
homeomorphism. Given two topological vector spaces $F_1$ and $F_2$, we
write $F_1\simeq F_2$ if there exists at least one topological
isomorphism from $F_1$ to $F_2$. This defines an equivalence relation
on the collection of all topological vector spaces. Given two
topological vector spaces $F_1$ and $F_2$, we endow the vector space
$F_1\times F_2=F_1\oplus F_2$ with the product topology, which makes it into a tvs.

Suppose that $F_1$ and $F_2$ are Fr{\'e}chet spaces and
$f:F_1\rightarrow F_2$ is a continuous map. Then $f$ is a topological
homomorphism iff $f$ has closed range.  Moreover, the open mapping
theorem states that any continuous and surjective linear map
$f:F_1\rightarrow F_2$ is open. In particular, $f$ is a topological
isomorphism iff it is continuous and bijective.

A Fr\'echet-Schwartz (FS) space is a Fr\'echet space which is also a
Schwartz space (see \cite{Grot,MV}). Every such space is
{\em reflexive}, i.e. naturally topologically isomorphic with the strong
topological dual of its strong topological dual. A tvs is called a
 DFS space if it is the strong topological dual of an FS space;
DFS spaces are also reflexive.

\

\begin{Definition}
Let $F_1$ and $F_2$ be topological vector spaces. A {\em topological
  pairing} between $F_1$ and $F_2$ is a bilinear map $\langle\cdot,
\cdot \rangle:F_1\times F_2\rightarrow \C$ which is (jointly) continuous.
\end{Definition}

\

\begin{Definition}
\label{def:perfect}
Let $\langle \cdot, \cdot \rangle:F_1\times F_2 \rightarrow \C$ be a
topological pairing. The {\em left Riesz morphism} of $\langle \cdot,
\cdot \rangle$ is the linear map $\tau_l:F_1\rightarrow F_2^\ast$ defined
through:
\be
\tau_l(u)(v)\eqdef \langle u, v\rangle~,~~\forall u\in F_1~,~\forall v \in F_2~~.
\ee 
The {\em right Riesz morphism} of
$\langle \cdot, \cdot \rangle$ is the linear map
$\tau_r:F_2\rightarrow F_1^\ast$ defined through:
\be
\tau_r(v)(u)\eqdef \langle u, v\rangle~,~~\forall u\in F_1~,~\forall v \in F_2~~.
\ee 
We say that $\langle \cdot, \cdot \rangle$ is a {\em perfect
  topological pairing} if both $\tau_l$ and $\tau_r$ are topological
isomorphisms.
\end{Definition}

\

\begin{remark} If $\langle\cdot, \cdot \rangle:F_1\times F_2\rightarrow \C$ is a
perfect topological pairing, then the topological vector spaces $F_1$
and $F_2$ are reflexive:
\beqa
&& (F_1^\ast)^\ast\simeq F_2^\ast\simeq F_1~~,\nn\\
&& (F_2^\ast)^\ast\simeq F_1^\ast\simeq F_2~~.
\eeqa
If we identify $F_2$ with $F_1^\ast$ using the right Riesz isomorphism and $F_2^\ast$ with $F_1$ using 
the left Riesz isomorphism, then the bidual of $F_1$ identifies with $F_1$ and 
the perfect pairing identifies with the duality pairing between $F_1$ and $F_1^\ast$. 
\end{remark}

\

\begin{Definition}
A {\em topological cochain complex} is a sequence:
\be
(F^\bullet):~~~ \ldots \longrightarrow F^{k-1}\stackrel{\updelta_{k-1}}{\longrightarrow} F^{k}\stackrel{\updelta_k}\longrightarrow F^{k+1}\stackrel{\updelta_{k+1}}{\longrightarrow} \ldots 
\ee
where $F^k$ are topological vector spaces and $\updelta_k$ are
continuous linear maps which satisfy $\updelta_{k+1}\circ
\updelta_k=0$ for all $k\in \Z$. The cohomology of
such a complex in degree $k\in \Z$ is the vector
space:
\be
\rH^k(F^\bullet,\updelta)\eqdef \ker \updelta_k /\im\updelta_{k-1}~~,
\ee
endowed with the quotient topology. We say that the topological
complex $F^\bullet$ is {\em bounded} if there exist integers $k_1<k_2$
such that $F^k=0$ unless $k_1\leq k\leq k_2$.
\end{Definition}

\

\noindent Given a bounded topological cochain complex $F^\bullet$, we
set $F\eqdef \oplus_{k\in \Z} F^k=F^{k_1}\times \ldots \times F^{k_2}$
(endowed with the direct product topology) and $\updelta=\sum_{k\in
\Z} \updelta_k=\sum_{k=k_1}^{k_2} \updelta_k$. Then $F$ is a finitely 
$\Z$-graded topological vector space and $\updelta$ is a continuous degree $+1$
endomorphism of $F$ which satisfies $\updelta^2=0$. The total
cohomology:
\be
\rH(F^\bullet,\updelta)=\ker \updelta/\im \updelta=\rH(F,\updelta)=\bigoplus_{k\in \Z} \rH^k(F,\updelta)=\rH^{k_1}(F,\updelta)\times \ldots \times \rH^{k_2}(F,\updelta)
\ee
is a finitely $\Z$-graded topological vector space. When $u\in F^k$,
we set $\deg u\eqdef k$.

\

\begin{Definition}
Let $(F,\updelta)$ be a bounded topological cochain complex. The {\em
  topological dual} of $(F,\updelta)$ is the topological cochain
complex $(F^\ast,\updelta^t)$ defined through:
\be
(F^\ast)^k\eqdef (F^{-k})^\ast~~,~~(\updelta^t)_k\eqdef \updelta_{-k-1}^t~~,
\ee
where $\updelta_j^t$ denotes the transpose of $\updelta_j$. 
\end{Definition}

\

\begin{Definition}
\label{def:top_pairing}
Let $(F,\updelta)$ and $(\hF, \hupdelta)$ be two bounded topological
cochain complexes. A $\C$-bilinear map $\langle \cdot, \cdot
\rangle:F\times {\hF}\rightarrow \C$ is called a {\em topological
  pairing of complexes} if it satisfies the following conditions:
\begin{enumerate}
\itemsep 0.0em
\item \,$\langle \cdot, \cdot \rangle$ is jointly continuous.
\item \,$\langle\cdot, \cdot \rangle$ has degree zero:
\be
\langle \cdot, \cdot\rangle|_{F^i\times \hF^{j}}=0~~\mathrm{if}~~i+j\neq 0~.
\ee
\item \,$\langle \updelta u, v\rangle+(-1)^{\deg u}\langle u,\hupdelta v\rangle=0~$
for all homogeneous elements $u\in F$ and $v\in {\hF}$. 
\end{enumerate}
In this case, we say that $\langle \cdot, \cdot \rangle$ is a {\em
  perfect topological pairing of complexes} if its
restriction to $F^k\times \hF^{-k}$ is perfect for all $k\in \Z$.
\end{Definition}

\

\noindent A topological pairing of bounded topological cochain
complexes induces a degree zero topological pairing $\langle \cdot,
\cdot \rangle^H:\rH(F,\updelta)\times \rH(\hF,\hupdelta)\rightarrow
\C$ between total cohomologies.

\

\begin{Definition}
A topological pairing of bounded topological cochain complexes 
is called {\em cohomologically perfect} if the restriction
$\langle \cdot, \cdot \rangle^H \big|_{\rH^k(F,\updelta)\times
  \rH^{-k}(\hF,\hupdelta)}$ is a perfect topological pairing for all
$k\in \Z$.
\end{Definition} 

\

\noindent When $\langle \cdot, \cdot \rangle$ is cohomologically
perfect, the vector spaces $\rH^k(F,\updelta)$ and
$\rH^k(\hF,\hupdelta)$ are reflexive for all $k\in \Z$ and
$\langle\cdot,\cdot\rangle^H$ induces topological isomorphisms
$\rH^k(F,\updelta)\simeq \rH^{-k}(\hF,\hupdelta)^\ast$.

\

\begin{Proposition} 
\label{prop:fdDual}
Let $(F,\updelta)$ be a bounded topological cochain complex of FS spaces such
that $\rH^k(F,\updelta)$ is finite-dimensional for all $k\in
\Z$. Then:
\begin{enumerate}
\itemsep 0.0em
\item $\updelta_k$ is a topological homomorphism for all $k\in \Z$.
\item The dual complex $(F^\ast, \updelta^t)$ is a finite topological cochain
  complex of DFS spaces whose differentials are topological
  homomorphisms and whose cohomology is finite-dimensional in every
  degree.
\item The natural linear map:
\be
\rH^{-k}(F^\ast, \updelta^t)\rightarrow \rH^{k}(F,\updelta)^\vee~~
\ee
is bijective for all $k\in \Z$.
\end{enumerate} 
\end{Proposition}

\begin{proof}
Follows from \cite[Theorem 1.5]{TL} and \cite[Theorem 1.6]{TL} upon
noticing that $\rH^k(F,\updelta)$ is separated and that its topological dual
coincides with the algebraic dual since $\rH^k(F,\updelta)$ is
finite-dimensional. \qed
\end{proof}

\

\begin{Corollary}
\label{cor:perf_criterion}
Let $(F,\updelta)$ and $(\hF, \hupdelta)$ be two bounded topological
complexes and let $\langle \cdot, \cdot \rangle:F\times
\hF\rightarrow\C$ be a perfect topological pairing between these
complexes. Suppose that $\rH^k(F,\updelta)$ is finite-dimensional for
all $k\in \Z$. Then $\rH^k(\hF,\hupdelta)$ is finite-dimensional for
all $k\in \Z$ and $\langle \cdot, \cdot \rangle$ is cohomologically
perfect.
\end{Corollary}

\begin{proof}
Follows immediately from Proposition \ref{prop:fdDual}.\qed
\end{proof}

\section{Some results on spectral sequences}
\label{sec:spec}

In this section, we use the notations and conventions of \cite{BoT}.
Let $K=\oplus_{p,q\in \Z} K^{p,q}$ be a double complex of $\C$-vector
spaces with vertical differential $\dd_1:K^{p,q}\to K^{p,q+1} $ and
horizontal differential $\dd_2:K^{p,q}\to K^{p+1,q}$. We are
interested in the following special cases, which will arise in later
sections:
\begin{enumerate}[A.]
\item $K$ is concentrated in the first quadrant, i.e. $K^{p,q}$
  vanishes unless $p>0$ and $q>0$.
\item $K$ is concentrated in a horizontal strip above the horizontal
  axis, i.e. there exists $N>0$ such that $K^{p,q}$ vanishes unless
  $0\leq q \leq N$.
\end{enumerate}
Let $K=\oplus_{n\in \Z} K^n$ be the decomposition corresponding to the total 
grading of $K$, where: 
\be
K^n\eqdef \bigoplus_{p+q=n}K^{p,q}~~.
\ee
Let $\updelta\eqdef \dd_1+\dd$ be the total differential, where
$\dd\eqdef (-1)^p \dd_2$. The double complex can be
endowed with the \emph{standard} decreasing filtration $F$ given by:
\ben
\label{f9}
F^p K\eqdef \bigoplus_{i\geq p} \bigoplus_{q\in \Z} K^{i,q}~~.
\een
This filtration can in general be unbounded (as in case
B. above). However, for any $n\in \Z$, the spaces:
\be
\gr_F^p(K^n)\eqdef \frac{[K^n\cap F^p K]}{[K^n\cap F^{p-1} K]}=
\big[\!\bigoplus_{i\geq p, \,q=n-i}\!\! K^{i,q}\big]\Big/ \big[\!\bigoplus_{i\geq p-1, \,q=n-i} \!\! K^{i,q}\big]\simeq_\C K^{p,n-p}
\ee 
vanish in both cases A. and B. except for a finite
number of values of $p$. Here and below, the symbol $\simeq_\C$
denotes isomorphism of $\C$-vector spaces. Applying the
theory of exact couples to the bigraded complex $K$ endowed with the
filtration $F$, we obtain a spectral sequence which computes the
cohomology $\rH_{\updelta}^n(K)$ of the total complex $(K,\updelta)$:

\

\begin{Proposition}
\label{pro8} 
Assume that the double complex $(K,\dd_1,\dd_2)$ satisfies either of
the conditions A. or B. above. Then the filtration (\ref{f9})
defines a spectral sequence $\bE=(\bE_r,\bd_r)_{r\geq 0}$ which
converges to the total cohomology $\rH_\updelta(K)\eqdef \oplus_{n\in
  \Z}\rH^n_{\updelta}(K)$. For each $r\geq 0$, the page $\bE_r$ is
endowed with a bigrading given by the decomposition
$\bE_r=\oplus_{p,q\in \Z}\bE_r^{p,q}$ and with a
differential\,\footnote{We explain later (see \eqref{f7} and
  \eqref{specdef}) the relation between $\bd_r$, 
$\dd_1$ and $\dd_2$.} $\bd_r: \bE^{p,q}_r\to \bE_r^{p+r,q-r+1}$ 
defined recurrently by:
\be
 \bE^{p,q}_{r} \eqdef \rH(\bE^{p,q}_{r-1},\bd_{r-1})~.
\ee
For the first pages we have $\bd_0=\dd_1$ and $\bd_1=\dd:=(-1)^p\dd_2$, hence:
\be
\bE_1^{p,q}=\rH_{\dd_1}^{q}(\gr_F^p K)~,~~\bE_2^{p,q}=\rH_{\dd}^{p}(\bE_1^{p,q})~~.  
\ee
For each $n\in \Z$, the filtration \eqref{f9} induces a decreasing
filtration $(\bF^p\rH^n_\updelta(K))_{p\in \Z}$ of the vector space
$\rH^n_\updelta(K)$, whose associated gradeds~ $\gr^p_\bF
\,\rH^n_\updelta(K)\eqdef \frac{\bF^p\rH^n_\updelta(K) }{
    \bF^{p-1}\rH^n_\updelta(K)}$ satisfy:
\be
\gr^p_\bF \,\rH^n_\updelta(K) \simeq_\C \bE^{p,n-p}_{\infty}~,~~\forall p\in \Z~~,
\ee
where $\bE_{\infty}=\oplus_{p,q\in\Z}\bE^{p,q}_{\infty}$ is the limit of $\bE$.
\end{Proposition}

\begin{proof} 
The proof can be found in \cite[\SSS 14]{BoT} (see Theorem 14.6 and
Theorem 14.14). The proof is not restricted to the case when $F^p K$ is a
finite filtration, but requires only that the induced filtration $F^p K^n$
is finite for each $n$, which is true when condition A. or condition B.
is satisfied. \qed
\end{proof}

Consider another double complex  $(\hK, \hd_1,\hd_2)$ with the total
differential $\hat{\updelta}=\hd_1+\hd$, where $\hd=(-1)^p\hd_2$.
Let $\tau: K\rightarrow \hK$ be a morphism of double complexes
and $\tau_*:\rH_\updelta(K)\rightarrow \rH_\hdelta(\hK)$ denote the
morphism of graded $\C$-vector spaces induced by $\tau$ on total cohomology.  Let
$F^p\hK$ denote the analogue of the filtration \eqref{f9} for $\hK$ and
$(\hbF^p \rH^n_{\hupdelta}(\hK))_{p\in \Z}$ denote the filtration induced by $F^p$ on
the homogeneous components of total cohomology.

\vspace{3mm}

\begin{Theorem}
\label{T10} Suppose that $\tau$ is injective and that it induces
isomorphisms of vector spaces $\rH_{\dd_1}^{p,q}(K)\simeq_\C
\rH_{\hd_1}^{p,q}(\hK)$ for all $p,q\in \Z$ in vertical cohomology.
Assume that both double complexes $(K,\dd_1,\dd_2)$ and
$(\hK,\hd_1,\hd_2)$ satisfy condition A. or that both satisfy
condition B. above. Then $\tau_\ast$ satisfies:
\be
\tau_\ast (\bF^p \rH^n_\updelta (K))\subset \hbF^p \rH^n_{\hdelta}(\hK)~,~\forall p, n\in \Z
\ee
and restricts to isomorphisms of vector spaces between the associated gradeds:
\be
\tau_\ast: \gr_\bF^p \,\rH^n_\updelta (K)\stackrel{\sim}{\rightarrow} \gr_{\hbF}^p \,\rH^n_{\hdelta}(\hK)~,~\forall p, n\in \Z~.
\ee
\end{Theorem}

\begin{proof} 
We apply Proposition \ref{pro8} for both double complexes
$(K,\dd_1,\dd_2)$ and $(\hK, \hd_1,\hd_2)$ and denote by $\bE$ and
$\bhE$ the spectral sequences defined by the standard filtration
$F^p$.  By assumption, the first pages of these two spectral sequences
coincide. The statement of the theorem follows if we show that the
spectral sequences coincide on each page. To show this, we have
to look closer at the components $E_r^{p,q}$, whose elements arise as
cohomology classes of previous pages.

An element $b\in K$ represents a cohomology class in $\bE_r$ iff it is
a cocycle in all $\bE_1$, $\bE_2$, \ldots, $\bE_{r-1}$. Let us denote
by $[b]_r$ the image of $b$ in $\bE_r$ if defined. An explicit
description of this can be found in \cite[page 164]{BoT}, which states
that $b\in K^{p,q}$ represents an element of $\bE_r$ iff there exists
a chain of elements $c_i\in K^{p+i,q-i}$ with $1\leq i\leq r-1$ such
that:
\ben
\label{f7}
\dd_1 b=0~,~~ \dd b=-\dd_1c_1~~\mathrm{and}~~ \dd c_{i-1}=-\dd_1c_i
\een
for $i=2,\ldots, r-1$. Moreover, the differential $\bd_r$ of $\bE_r$ is given by:
\ben 
\label{specdef}
\bd_r [b]_r \eqdef [\dd c_{r-1}]_r~.
\een
The differential $\bd_r[b]_r$ is defined for $b$ belonging to $\ker
\dd_1$ and depends only on the class $[b]_1\in
\rH_{\dd_1}^{p,q}(K)$. Indeed, for any $b'\in K^{p,q-1}$, the element
$\tilde b:=b+\dd_1b'$ satisfies $\dd_1 \tilde b=0$, $\dd\tilde{b}=\dd b 
+\dd (\dd_1b')=-\dd_1c_1-\dd_1\dd b'=-\dd_1(c_1+\dd b')$. We also have
$\dd (c_1+\dd b')=\dd c_1$. This means that $\tilde b$ satisfies
the same system of equations (\ref{f7}) after modifying $\tilde
c_1:=c_1+\dd b'$. Thus, for any $\tilde{b}$ such that $[\tilde
  b]_1=[b]_1\in H_{\dd_1}^{p,q}(K)$, we have $\bd_r[\tilde{b}]_r
=[\dd c_{r-1}]_r =\bd_r[b]_r$. The definition of $\bd_r$ does not
depend on the choice of the elements $c_1,\ldots, c_{r-1}$ 
representing classes $[c_{i-1}]_{i}\in \bE_i$ and satisfying
\eqref{f7}. Everything said above also holds for the differentials
$\bhd_r$ of the spectral sequence defined by the filtration $F$
of the complex $\hK^{p,q}$.  We will also use the notation
$\tau_*$ for the map induced by $\tau$ on any of $\bE^{p,q}_r$.

\

\noindent In what follows we prove by induction on $s$ that the following statements hold 
for any $p,q\in \Z$:
\begin{itemize}
\item[(\textbf{i})] the inclusion $\tau$ induces an isomorphism
  $\tau_*: \bE^{p,q}_s(K)\stackrel{\sim}{\rightarrow}
  \bhE^{p,q}_s(\hK)$,
\item[(\textbf{ii})] $[\tau b]_s=\tau_*[b]_s$ for any $b\in K^{p,q}$
  such that $[b]_s$ is defined,
\item[(\textbf{iii})] $\bhd_{s}[\hat{b}]_s=\bhd_{s}[\tau_*
  [b]_1]_s=\tau_* \bd_s[b]_s$ for any $\hat{b}\in K^{p,q}$ and for $b
  \in K^{p,q}$ such that $[\tau b]_1=[\hat{b}]_1$ and $[b]_s$ is
  defined.
\end{itemize}
For $s=1$, statement (\textbf{i}) is just the assumption of the theorem
while the other two statements follow from the fact that $\tau$ is an injective map of
double complexes. 

Let us assume that (\textbf{i}) -- (\textbf{iii}) hold for all $s\leq r-1$.
 To show (\textbf{ii}) for $s=r$, consider an element $b\in K^{p,q}$ such that
$[b]_r$ is defined. Then we have $\bd_{r-1}[b]_{r-1}=0$ as well as $\bhd_{r-1}[\tau
  b]_{r-1}=0$. We also  have $[\tau  b]_{r-1}=\tau_*[b]_{r-1}$. Using
(\textbf{i}) and (\textbf{iii}) for $s=r-1$ gives:
\be
 [\tau b]_{r-1}+\bhd_{r-1}\bhE_{r-1} =
\tau_*[b]_{r-1}+\bhd_{r-1}\bhE_{r-1}=
\tau_*([b]_{r-1}+\bd_{r-1}\bE_{r-1})~~,
\ee 
which shows that $[\tau b]_{r}=\tau_*[b]_{r}$ as cohomology classes
in $\bE_r$. To show that (\textbf{iii}) holds for $s=r$, consider an
element $\hat{b}\in \hK^{p,q}$ which represents $[\hat{b}]_r$ in
$\bhE_r$. Since $\tau_\ast: \rH_{\dd_1}^{p,q}(K)\rightarrow
\rH_{\hd_1}^{p,q}(\hK)$ is an isomorphism, there exist $b\in K^{p,q}$
such that $\tau_*[b]_1=[\hat{b}]_1$. A chain of elements $c_i\in
K^{p+i,q-i}$ ($1\leq i\leq r-1$) satisfying (\ref{f7}) defines
$\bd_r[b]_r$. Since $\tau$ is a morphism of bigraded complexes, we
have $\hd_1(\tau c )= \tau (\dd_1 c)$ and $\hd (\tau c) = \tau (\dd
c)$. Then the chain of elements $\tau c_i\in \hK^{p-i,q+i}$ and
$\tau b$ satisfy \eqref{f7} for $\bhE_r$, so it determines
$\hat{\bd}_r[\hat{b}]_r$. To see this, we compute:
\begin{multline}
\hat{\bd}_r[\hat{b}]_r=\bhd_r\big[[\hat{b}]_1\big]_r = \bhd_r\big[\tau_*[b]_1\big]_r=\bhd_r\big[[\tau b]_1\big]_r
=[\dd (\tau c_{r-1})]_r=[\tau(\dd c_{r-1})]_r=\\ =\tau_*[\dd c_{r-1}]_r=\tau_*(\bd_r\big[[b]_1\big]_r) = \tau_*(\bd_r[b]_r) ~ ,\quad\quad\quad
\end{multline}
where we used statement (\textbf{ii}) for $b:=c_{r-1}$.

To show that (\textbf{i}) holds for $s=r$, we start from statement
(\textbf{iii}) for $s=r-1$, which tells us that the differentials are
compatible with the isomorphism $\tau_*$, namely we have
$\hat{\bd}_{r-1}(\tau_*[b])=\tau_*(\bd_{r-1}[b])$. This implies that
the cohomology taken with respect to these differentials is also
compatible with $\tau_*$. Thus:
\be
\bhE^{p,q}_r=\rH(\bhE^{p,q}_r,\hat{\bd}_{r-1}) =
\rH\big(\tau_*({\bE}^{p,q}_r),\hat{\bd}_{r-1}\big) =
\tau_*\big(\rH(\bE^{p,q}_r,\bd_{r-1})\big)=\tau_*\big(\bE^{p,q}_r\big) ~.
\ee 
Since the bigradings of both graded complexes $K$ and $\hK$ satisfy
either condition A. or condition B., it follows that for any $n$ there
exists $N=N(n)$ such that $\bd_r (K^{p',q'})=0=\hat{\bd}_r
\big(\hK^{p',q'}\big)$ for all $r>N$, $p+q = n$ and for both sets of
pairs $(p',q')=(p+r,q-r+1)$ and $(p',q')=(p-r,q+r-1)$. Thus for $r>N$
and $p+q=n$, we have
$\tau_*(\bE^{p,q}_{\infty})=\tau_*(\bE^{p,q}_r)=\bhE^{p,q}_r=\bhE^{p,q}_{\infty}$. By
Proposition \ref{pro8}, this is equivalent with the statement that the
map $\tau_*:\gr_\bF^p\,\rH^n_\updelta (K)\stackrel{\sim}{\rightarrow}
\gr_\hbF^p\,\rH^n_{\hat{\updelta}}(\hK)$ is an isomorphism for any $p$
and $n$. \qed
\end{proof}

\section{Graded Serre pairings}
\label{sec:graded}

In this section, we discuss a version of the Serre pairing (see
\cite{Serre}) which exists for graded holomorphic vector bundles.
This pairing will arise later on in the proof of non-degeneracy of
cohomological bulk and boundary traces.

\subsection{Topological complexes of differential forms valued in a holomorphic vector bundle}

For any holomorphic vector bundle $V$ on $X$, let
$\rOmega^{p,q}(X,V)=\rGamma_\sm(X,\wedge^p T^\ast X\otimes
\wedge^q \bar{T}^\ast X\otimes V)$ denote the space of $V$-valued smooth forms of
type $(p,q)$ defined on $X$ and $\rOmega^{p,q}_c(X,V)$ denote the
subspace of $\rOmega^{p,q}(X, V)$ consisting of compactly-supported
forms. Then $\rOmega^{p,q}(X,V)$ is an FS space
\cite{Serre,Cassa} when endowed with the topology of uniform
convergence of all derivatives on compact subsets. The subspace
$\rOmega^{p,q}_c(X,V)$ is dense in $\rOmega^{p,q}(X,V)$, being the
space of ``test sections'' of the bundle $\wedge^p T^\ast X\otimes \wedge^q
\bar{T}^\ast X\otimes V$. Let:
\be
\rOmega(X,V)\eqdef \bigoplus_{p,q=0}^d \rOmega^{p,q}(X,V)~~,~~\rOmega_c(X,V)\eqdef \bigoplus_{p,q=0}^d \rOmega^{p,q}_c(X,V)
\ee
and
\be
\cA(X,V)\eqdef \bigoplus_{q=0}^d \rOmega^{0,q}(X,V)~~,~~\cA_c(X,V)\eqdef \bigoplus_{q=0}^d \rOmega^{0,q}_c(X,V)~~.
\ee
Then $\rOmega(X,V)=\rGamma_\sm(X, \wedge
T^\ast X\otimes\wedge\bar{T}^\ast X\otimes V)$ and $\cA(X,V)=\rGamma_\sm(X,\wedge \bar{T}^\ast X\otimes V)$ are FS spaces which contain $\rOmega_c(X,V)$
(respectively $\cA_c(X,V)$) as dense subspaces.  Notice that
$\cA(X,V)$ is a closed subspace of $\rOmega(X,V)$. Moreover,
$(\cA(X,V),\bpd_V)$ is a finite topological complex of FS spaces
\cite{Serre,Cassa}, where $\bpd_V$ denotes the Dolbeault differential
of $V$.

\subsection{Topological complexes of bundle-valued currents with compact support and the classical Serre pairing} 

Let $\hOmega^{p,q}(X,V)$ denote the space of distributions
with compact support valued in the bundle
$ \wedge^p T^\ast X\otimes \wedge^q {\bar T}^\ast X\otimes V$.
Consider the bigraded topological vector space:
\be
\hOmega(X,V)\eqdef \bigoplus_{p,q=0}^d \hOmega^{p,q}(X,V)~~
\ee
of distributions with compact support valued in the vector bundle
$ \wedge T^\ast X\otimes \wedge \bar{T}^\ast X\otimes V$.  Let $V^\vee\eqdef
Hom(V,\cO_X)$ denote the dual bundle to $V$. Then
$\hOmega^{d-p,d-q}(X,V)$ is topologically isomorphic (see \cite{Serre}) with
the topological dual $\rOmega^{p,q}(X,V^\vee)^\ast$, where the latter
is endowed with the strong topology. The corresponding perfect duality pairing
is known as the {\em Serre pairing} and it is given by (see \cite{Serre}):
\be
(\omega,T)\rightarrow \int_X\omega\wedge T~~,~\forall \omega\in \rOmega^{p,q}(X,V^\vee)~,~\forall T\in \hOmega^{d-p,d-q}(X,V)~~,
\ee
where $\int_X$ denotes integration of compactly supported currents of
type $(d,d)$ on $X$ with respect to the orientation induced by the
complex structure of $X$.  Below, we introduce a version of this
pairing adapted to the case when $V$ is replaced by a $\Z$-graded or
$\Z_2$-graded holomorphic vector bundle.

\subsection{The graded Serre pairing of a $\Z$-graded or $\Z_2$-graded holomorphic vector bundle}
\label{subsec:SerrePairing}

Let $A$ be either of the Abelian groups $\Z$ or $\Z_2$.  Let
$Q=\oplus_{j\in A} Q^j$ be an $A$-graded holomorphic vector 
bundle\footnote{For $A=\Z$, the grading of $Q$ is necessarily
  concentrated in a finite number of degrees, since $Q$ has finite
  rank.}
defined on $X$.  Let $Q^\vee$ denote the dual vector bundle of $Q$,
which we grade by the decomposition:
\be
Q=\bigoplus_{j\in A}{(Q^\vee)^j}~~,
\ee 
where $(Q^\vee)^j\eqdef (Q^{-j})^\vee$. In this case, the bundles
$\wedge T^\ast X \otimes \wedge \bar{T}^\ast X \otimes Q$ and $\wedge
T^\ast X \otimes \wedge \bar{T}^\ast X \otimes Q^\vee$ are
$\Z^2\times A$-graded with homogeneous components:
\beqa
&& (\wedge T^\ast X \otimes \wedge \bar{T}^\ast X \otimes Q)^{p,q,i}\eqdef \wedge^p T^\ast X \otimes \wedge^q \bar{T}^\ast X \otimes Q^i~~,\\
&& (\wedge T^\ast X \otimes \wedge \bar{T}^\ast X \otimes Q^\vee)^{p,q,j}\eqdef \wedge^p T^\ast X \otimes \wedge^q \bar{T}^\ast X \otimes (Q^{-j})^\vee~~.
\eeqa
Viewing $\cO_X$ as an $A$-graded holomorphic vector bundle
concentrated in degree zero, the bundle $\wedge T^\ast X \otimes
\wedge \bar{T}^\ast X\simeq \wedge T^\ast X \otimes \wedge
\bar{T}^\ast X \otimes \cO_X$ is also $\Z^2\times A$-graded with the
third grading concentrated in degree zero. The spaces
$\rOmega(X,Q)=\rGamma_\sm(X, \wedge T^\ast X \otimes \wedge
\bar{T}^\ast X\wedge Q)$ and $\hOmega(X,Q)$ are trigraded accordingly. Notice 
that $\rOmega(X,Q)$ is an FS space, while $\hOmega(X,Q)$ is a DFS space. 

\

\begin{Definition} 
The {\em graded duality morphism of $Q$} is the morphism
$\ev_Q:Q\otimes Q^\vee \rightarrow \cO_X$ of $A$-graded holomorphic
vector bundles determined uniquely by the condition:
\ben
\label{ev_Q}
\ev_Q(x)(v)(w)\eqdef (-1)^i \delta_{i+j,0} w(v)~,~~\forall v\in  Q^i~,~\forall w\in  (Q^\vee)^j~,~\forall x\in X~~.
\een
\end{Definition}

\

\noindent Together with the wedge product of differential forms, $\ev_Q$ induces a
morphism of holomorphic vector bundles:
\ben
\label{cBdef}
\cS_Q:(\wedge T^\ast X \otimes \wedge \bar{T}^\ast X \otimes Q)\otimes 
(\wedge T^\ast X\otimes \wedge \bar{T}^\ast X \otimes Q^\vee)\rightarrow \wedge T^\ast X\otimes \wedge \bar{T}^\ast X
\een
which is determined uniquely by the condition: 
\ben
\label{cBQdef}
\cS_Q(x)(\omega_1\otimes v, \omega_2\otimes w)\eqdef (-1)^{i(p_2+q_2+1)} \updelta_{i+j,0}w(v) \, \omega_1\wedge \omega_2  
\een
for $\omega_1\in \wedge^{p_1} T_x^\ast X\otimes \wedge^{q_1} \bar{T}^\ast_x
X$, $\omega_2\in \wedge^{p_2} T_x^\ast X\otimes \wedge^{q_2} \bar{T}^\ast_x
X$ and $v\in Q_x^i, w\in (Q_x^\vee)^j$ (where $x\in X$).  The bundle
morphism $\cS_Q$ induces a continuous bilinear map: 
\ben
\label{BQdef}
S_Q:\rOmega(X,Q)\times \hOmega(X,Q^\vee)\rightarrow \hOmega(X)~~. 
\een
With respect to the trigradings described above, the maps
$\cS_Q$ and $S_Q$ are homogeneous of tridegree $(0,0,0)$.

\

\begin{Definition}
The graded Serre pairing of the $A$-graded holomorphic vector bundle $Q$
is the topological pairing $\S_Q:\rOmega(X,Q)\times
\hOmega(X,Q^\vee)\rightarrow \C$ defined through:
\be
\S_Q(\omega,T)\eqdef \int_X S_Q(\omega, T)~~,
\ee
where $\int_X L$ is defined to equal zero unless $L\in \hOmega(X)$ has type $(d,d)$. 
\end{Definition}

\

\noindent If we view $\C$ as a $\Z^2\times A$-graded vector space whose grading is
concentrated in degree $(0,0,0)$, then $\S_Q$ has tridegree
$(-d,-d,0)$. Thus:
\be
\S_Q(\omega,T)=\updelta_{p+p'\!,d}\updelta_{q+q'\!,d}\updelta_{i+j,0}\S_Q(\omega,T)~,
\ee
when $\omega\in \rOmega^{p,q}(X,Q^i)$ and $T\in \hOmega^{p'\!,q'}(X,(Q^\vee)^j)$.

\

\begin{Lemma}
\label{lemma:Serre}
$\S_Q$ is a perfect pairing  between the topological vector spaces
  $\rOmega(X,Q)$ and $\hOmega(X,Q^\vee)$.
\end{Lemma}

\begin{proof}
Follows immediately from \cite[Proposition 4]{Serre} and \cite[Proposition 5]{Serre}. \qed
\end{proof}

\section{Non-degeneracy of the bulk trace}
\label{sec:bulk}

In this section, we prove non-degeneracy of the bulk trace defined in
\cite{nlg1}. The proof uses the spectral sequence results obtained in
Section \ref{sec:spec} together with an adaptation of the argument of
\cite{Serre} to the case of $\Z$-graded holomorphic vector bundles.

\subsection{The topological complex of polyvector-valued forms}

Let $(X,W)$ be a holomorphic Landau-Ginzburg pair. Consider the
cochain complex $(\PV(X),\updelta_W)$, where $\updelta_W=\bpd_{\wedge
TX}+\ioda_W$ is the twisted differential \eqref{updelta}, graded by
the total $\Z$-grading:
\be
0\longrightarrow \PV^{-d}(X)\stackrel{\updelta_W}{\longrightarrow} \PV^{-d+1}(X)\stackrel{\updelta_W}{\longrightarrow}\ldots \stackrel{\updelta_W}{\longrightarrow}\PV^{d-1}(X)\stackrel{\updelta_W}{\longrightarrow} \PV^{d}(X)\longrightarrow 0~~.
\ee
Notice that $\updelta_W$ has total degree $+1$ with respect to this
grading. Let $\PV_c(X)$ denote the subcomplex of $(\PV(X),\updelta_W)$
formed by compactly-supported forms valued in the holomorphic vector
bundle $\wedge TX$. Notice that $\ioda_W$ and $\updelta_W$ are
continuous with respect to the Fr\'echet topology, since they are
differential operators of order zero and one, respectively. In
particular, $(\PV(X),\updelta_W)$ is a finite topological cochain
complex of FS spaces, which contains $\PV_c(X)$ as a dense subcomplex.

\subsection{The topological complex of compactly-supported polyvector-valued currents}

For any $i\in \{-d,\ldots, 0\}$ and any $j\in \{0,\ldots, d\}$, let
$\hPV^{i,j}(X)\eqdef \hOmega^{0,j}(X,\wedge^{|i|} TX)$. Consider the
bigraded vector space:
\be
\hPV(X)\eqdef \bigoplus_{i=-d}^0\bigoplus_{j=0}^d \hPV^{i,j}(X)=\bigoplus_{i=-d}^0\bigoplus_{j=0}^d \hOmega^{0,j}(X,\wedge^{|i|} TX)~~.
\ee
We endow $\hPV(X)$ with its  total $\Z$-grading, which has homogeneous components:
\be
\hPV^k(X)=\bigoplus_{i+j=k}\hPV^{i,j}(X)~~.
\ee
This grading is concentrated in degrees $k\in \{-d,\ldots,d\}$. Let
$\hupdelta_W:\hPV(X)\rightarrow \hPV(X)$ be the natural extension of
$\updelta_W$ to $\hPV(X)$. Then $(\hPV(X),\hupdelta_W)$ is a topological cochain
complex of DFS spaces:
\be
0\longrightarrow \hPV^{-d}(X)\stackrel{\hupdelta_W}{\longrightarrow} \PV^{-d+1}(X)
\stackrel{\hupdelta_W}{\longrightarrow}\ldots \stackrel{\hupdelta_W}{\longrightarrow}\hPV^{d-1}(X)
\stackrel{\hupdelta_W}{\longrightarrow} \hPV^{d}(X)\longrightarrow 0~~.
\ee
\noindent 
Let $\hHPV^k(X,W)\eqdef \rH^k(\hPV(X),\hupdelta_W)$ denote the
cohomology of this cochain complex in degree $k\in \{-d,\ldots,d\}$.
Notice that $(\PV_c(X),\updelta_W)$ is naturally a subcomplex of
$(\hPV(X),\hupdelta_W)$.  As explained in \cite{nlg1}, the wedge
product induces an associative and supercommutative multiplication on
$\PV(X)$, which we denote by juxtaposition and which makes
$(\PV(X),\updelta_W)$ into a supercommutative differential graded
algebra. This multiplication operation is jointly continuous with
respect to the Fr\'echet topology, since the wedge product is.

\subsection{The canonical off-shell bulk pairing and its extension}

Let $\Omega$ be a holomorphic volume form on $X$. The {\em canonical
  off-shell bulk trace} \cite[Section 5]{nlg1} determined by $\Omega$ is
the continuous $\C$-linear map $\Tr_B:\PV_c(X)\rightarrow \C$ defined through:
\be
\Tr_B(\omega)\eqdef \int_X\Omega\wedge (\Omega\lrcorner \omega)~,~~\forall \omega\in \PV_c(X)~~,
\ee
where the integral over $X$ of a form of type $(k,l)$ is defined to
vanish unless $k=l=d$.  For simplicity, we do not indicate the
dependence of $\Omega$ in the notation $\Tr_B$. This map has bidegree
$(d,-d)$ and hence is of degree zero with respect to the total
$\Z$-grading on $\PV_c(X)$. Notice that $\Tr_B$ can be viewed as a
distribution valued in the holomorphic vector bundle $\wedge^d
\bar{T}^\ast X \otimes \wedge^d TX$. It can also be viewed as a
current of type $(d,0)$ valued in the holomorphic vector bundle
$\wedge^d TX$.

\

\begin{Definition}
The {\em canonical off-shell bulk pairing} determined by $\Omega$ is
the continuous bilinear map $\langle \cdot, \cdot \rangle_B:
\PV(X)\times \PV_c(X)\rightarrow \C$ defined through:
\be
\langle \omega, \eta\rangle_B \eqdef \Tr_B(\omega\eta)=\int_X\Omega\wedge 
[\Omega\lrcorner (\omega\eta)]~,~~\forall \omega\in \PV(X)~,~\forall \eta\in \PV_c(X)~~.
\ee
\end{Definition}

\noindent Notice that $\langle\omega, \eta \rangle_B$ is well-defined
since $\omega\eta$ belongs to $\PV_c(X)$. This pairing has degree zero 
when $\PV(X)$ and $\PV_c(X)$ are endowed with the total $\Z$-gradings. 

\

\begin{Definition}
The {\em extended canonical off-shell bulk pairing} is the continuous
bilinear map $\langle \cdot, \cdot \rangle: \PV(X)\times
\hPV(X)\rightarrow \C$ defined through:
\ben
\label{extpairing}
\langle \omega, T\rangle=\int_X\Omega\wedge [\Omega\lrcorner (\omega T)]~~,~~\forall \omega\in \PV(X)~,~\forall T\in \hPV(X)~~.
\een
\end{Definition}

\

\noindent We have $\langle \cdot, \cdot\rangle |_{\PV(X)\times
  \PV_c(X)}=\langle\cdot, \cdot\rangle_B$.  The pairing $\langle
\cdot, \cdot \rangle$ has degree zero when $\PV(X)$ and $\hPV(X)$ are
endowed with the total $\Z$-gradings.

\

\begin{Proposition}
\label{prop:pcomplexes}
The canonical off-shell bulk pairing $\langle \cdot , \cdot \rangle_B$
is a topological pairing of bounded $\Z$-graded complexes between
$(\PV(X),\updelta_W)$ and $(\PV_c(X),\updelta_W)$, while the extended
canonical off-shell bulk pairing $\langle \cdot, \cdot \rangle$ is a
topological pairing of bounded $\Z$-graded complexes between
$(\PV(X),\updelta_W)$ and $(\hPV(X),\updelta_W)$.
\end{Proposition}

\begin{proof} 
For $\omega\in \PV^k(X)$ and any $\eta\in \PV_c(X)$, we have:  
\be
\langle \updelta_W \omega,\eta\rangle_B+(-1)^k\langle \omega,\updelta_W\eta\rangle_B
=\Tr_B\left[(\updelta_W \omega)\eta+(-1)^{k}\omega(\updelta_W\eta)\right]=\Tr_B[\updelta_W(\omega\eta)]=0~~,
\ee
where in the last equality we used the property $\Tr_B\circ
\updelta_W=0$ (see \cite[Section 5]{nlg1}). This shows that
$\langle\cdot,\cdot\rangle_B$ is a paring of bounded complexes. Continuity of
this pairing follows from continuity of $\Tr_B$ and joint continuity
of the multiplication in $\PV(X)$. A formally similar argument shows
that $\langle\cdot, \cdot\rangle$ is a topological pairing of
bounded complexes. \qed
\end{proof} 

\

\noindent Let us view $\wedge T^\ast X$ as a $\Z$-graded holomorphic
vector bundle with $\wedge^k T^\ast X$ sitting in degree $+k$ and
$\wedge T X$ as a $\Z$-graded holomorphic vector bundle with $\wedge^k
TX$ sitting in degree $-k$.  Let $\Omega\lrcorner: \wedge T
X\rightarrow \wedge T^\ast X$ be the degree $d$ map of graded
holomorphic vector bundles given by contraction with $\Omega$. Let
$\Omega\lrcorner_0: \wedge T X\rightarrow \wedge T^\ast X$ be the map
of holomorphic vector bundles given by reduced contraction with
$\Omega$ (see \cite[Section 3]{nlg1}).  Let $\ev:=\ev_{\wedge TX}:\wedge
TX\otimes \wedge T^\ast X\rightarrow \cO_X$ denote the graded duality
morphism of the $\Z$-graded holomorphic vector bundle $\wedge TX$ (see
\eqref{ev_Q}).

\

\begin{Proposition}
\label{prop:evlr}
For any $x\in X$ and any $v_1, v_2 \in \wedge T_xX$, we have: 
\be
\Omega_x\lrcorner_0 (v_1\wedge v_2)=(-1)^{k_1 d}\ev_{\wedge TX}(x)(v_1, \Omega_x\lrcorner v_2)~~.
\ee 
\end{Proposition}

\begin{proof}
Suppose that $v_1\in \wedge^{k_1} T_xX$ and $v_2\in \wedge^{k_2} T_xX$. Then:
\beqa
\ev_{\wedge TX}(x)(v_1, \Omega_x\lrcorner v_2)&=&(-1)^{k_1}\delta_{k_1+k_2,d}(\Omega_x\lrcorner v_2)\lrcorner v_1=(-1)^{k_1}\delta_{k_1+k_2,d}
\Omega_x\lrcorner (v_2\wedge v_1)=\nn\\
&=&(-1)^{k_1+k_1k_2}\delta_{k_1+k_2,d}\Omega_x\lrcorner (v_1\wedge v_2)=(-1)^{k_1 d}\Omega_x\lrcorner_0 (v_1\wedge v_2)~~.\qed
\eeqa
\end{proof}

\

\noindent Contraction and wedge product with $\Omega$ induce topological isomorphisms:
\beqan
\label{top_isoms}
&&\Omega\lrcorner: \hPV(X)\rightarrow \hOmega^{0,\bullet}(X,\wedge T^\ast X)~~,\nn\\
&&\Omega\wedge: \PV(X)\rightarrow \rOmega^{d,\bullet}(X, \wedge T X)~~.
\eeqan

\begin{Proposition}
\label{prop:pairings}
For any $\Z$-homogeneous element $\omega\in \PV(X)$ and any $T\in \hPV(X)$, we have: 
\ben
\label{ident}
\langle \omega, T\rangle=(-1)^{d\, \deg \omega}\S_{\wedge TX}(\Omega\wedge \omega,\Omega\lrcorner T)~~.
\een
\end{Proposition}

\begin{proof}
It suffices to consider the case $\omega=\alpha\otimes v$ and
$T=\beta\otimes w$ with $\alpha \in \rOmega^{0,p}(X)$, $\beta\in
\hOmega^{0,q}(X)$, $v\in \rGamma_\sm(X,\wedge^{i} TX)$ and $w\in
\rGamma_\sm(X,\wedge^{j} TX)$. Then both sides of \eqref{ident} vanish
unless $p+q=d$ and $i+j=d$. When these conditions are satisfied,
we have:
\beqa
&& \S_{\wedge TX}(\Omega\wedge (\alpha\otimes v), \Omega\lrcorner(\beta\otimes w))=
(-1)^{q d}\S_{\wedge TX}((\Omega\wedge \alpha)\otimes v, \beta\otimes (\Omega\lrcorner w))=\nn\\
&=& (-1)^{q j}\int_X (\Omega\wedge \alpha \wedge \beta) \, \ev_{\wedge TX}(v,\Omega\lrcorner w)=
(-1)^{q j+d i}\int_X \Omega\wedge (\alpha\wedge \beta) \wedge \Omega\lrcorner_0(v\wedge w)=\nn\\
&=&(-1)^{q j+d^2+d i}\!\int_X\!\!\Omega\wedge \Omega\lrcorner_0((\alpha\wedge\beta)\otimes (v\wedge w))=
(-1)^{p j+qi}\!\int_X \!\!\Omega\wedge \Omega\lrcorner_0(\omega \wedge \eta)=(-1)^{d\, \deg \omega}\langle \omega, T\rangle~~
\eeqa
where we used Proposition \ref{prop:evlr} and noticed that:\footnote{We use the notation $\equiv_2$ for congruence modulo 2.} 
\be
q j+d(d+i)\equiv_2 (q+d)j\equiv_2 pj~~,~~pj+qi= p(d-i)+(d-p)i\equiv_2 d(p+i)\equiv_2 d\, \deg \omega~~. \qed
\ee
\end{proof}

\

\begin{Proposition}
\label{prop:extpairing}
The extended canonical off-shell bulk pairing $\langle\cdot, \cdot
\rangle$ is a perfect topological pairing between the topological
complexes $(\PV(X),\updelta_W)$ and $(\hPV(X),\hupdelta_W)$.
\end{Proposition}

\begin{proof}
The fact that $\langle \cdot, \cdot\rangle$ is a pairing of complexes
follows from Proposition \ref{prop:pcomplexes}. The fact that it is a
topological pairing follows from Proposition \ref{prop:pairings} and Lemma
\ref{lemma:Serre} applied to the graded holomorphic vector bundle
$Q:=\wedge TX$, using the fact that operations \eqref{top_isoms} are
topological isomorphisms and the existence of the natural isomorphism of holomorphic
vector bundles $Q^\vee\simeq \wedge T^\ast X$. \qed
\end{proof}

\subsection{Non-degeneracy of the cohomological bulk pairing}

\

\

\begin{Lemma}
\label{lemma:bulk}
Suppose that the critical set $Z_W$ is compact. Then $\HPV^k(X,W)$ and
$\hHPV^k(X,W)$ are finite-dimensional for every $k\in \{-d,\ldots,
d\}$ and the extended canonical off-shell bulk pairing $\langle \cdot,
\cdot \rangle$ determined by any holomorphic volume form $\Omega$ 
is cohomologically perfect. In particular, we have natural isomorphisms 
of $\C$-vector spaces:
\be
\hHPV^k(X,W)\simeq_\C \HPV^{-k}(X,W)^\vee~~,~~\forall k\in \{-d,\ldots, d\}~~
\ee
which depend only on $\Omega$. 
\end{Lemma}

\begin{proof}
Since $Z_W$ is compact, the cohomology
$\HPV^k(X,W\!)\!=\!\rH^k(\PV(X),\updelta_W\!)$ is finite-dimensional for every
$k$ as shown in \cite[Proposition 3.4]{nlg1}. The remaining statements
follow from Proposition \ref{prop:extpairing} and Proposition
\ref{prop:fdDual}. \qed
\end{proof}

\begin{remark}
Any two holomorphic volume forms $\Omega$ and $\Omega'$ on $X$ are related by: 
\be
\Omega'=f\Omega~~,
\ee
where $f:X\rightarrow \C^\times$ is a nowhere-vanishing holomorphic
function\footnote{Notice that such a function need not be constant
  since $X$ is non-compact.}  , i.e. an element of the group of units
$\O(X)^\times$ of the commutative ring $\O(X)$. If $\Tr'_B$ denotes
the canonical off-shell bulk trace determined by $\Omega'$, this
relation gives:
\be
\Tr'_B(\omega)=\int_X f^2\Omega\wedge (\Omega\lrcorner \omega)=\Tr_B(f^2\omega)~~,~\forall \omega\in \PV(X)~~.
\ee
Since $\updelta_W$ is $\O(X)$-linear, the space $\HPV(X,W)$ is a
$\Z$-graded $\O(X)$-module (which is finite-dimensional when $Z_W$ is
compact). Since $f$ is holomorphic, multiplication with $f^2$ commutes
with $\updelta_W$ and hence descends to an $\O(X)$-linear endomorphism
of the graded $\O(X)$-module $\HPV(X,W)$. We thus have the {\em squaring actions}:
\be
\sq_k:\O(X)^\times \rightarrow \Aut_{\O(X)}(\HPV^k(X,W))~,~~\forall k\in \{-d,\ldots, d\}~~
\ee
defined by $\sq_k(f)([\omega])\eqdef [f^2\omega]$ (where $[\omega]$
denotes the $\updelta_W$-cohomology class of a $\updelta_W$-closed
element of $\PV^k(X,W)$) and the cohomological bulk traces determined
by $\Omega$ and $\Omega'$ are related through:
\be
\Tr'=\Tr\circ \sq_0(f^2)~~.
\ee
\end{remark}

\

\noindent Let $s:\PV_c(X)\rightarrow \hPV(X)$ be the inclusion map and
$s_\ast:\HPV_c(X,W)\rightarrow \hHPV(X,W)$ be the linear map induced
by $s$ on cohomology. Recall the sheaf Koszul complex \eqref{cKdef}. 

\

\begin{Proposition}
\label{prop:ssbulk}
Suppose that the critical set $Z_W$ is compact. Then $\HPV_c(X,\!W\!)$ and
$\hHPV(X,\!W\!)$ are finite-dimensional over $\C$ and $s_\ast$ is an
isomorphism of $\C$-vector spaces. Moreover, for any $k\in
\{-d,\ldots, d\}$, we have a natural isomorphism of vector spaces:
\be
\hHPV^k(X,W)\simeq_\C \H^k_c(\cK_W)~~,
\ee
where $\H_c(\cK_W)$ denotes compactly-supported hypercohomology of $\cK_W$. 
\end{Proposition}

\begin{proof}
Since the category of all sheaves is Abelian with enough injectives
and $\cK_W$ is a finite complex, we can define the hypercohomology
$\H_c(\cK_W)$ (see \cite[Proposition 8.6]{Voisin}). In fact,
hypercohomology can be computed by using $\rGamma$-acyclic resolutions
(see \cite[Proposition 8.12]{Voisin}). Both $(\PV_c(X),\updelta_W)$
and $(\widehat{\PV}(X),\hupdelta_W)$ are complexes of fine (and thus
$\rGamma$-acyclic) sheaves (see \cite[Proposition 4.36]{Voisin}). Since
they both give $\rGamma$-acyclic resolutions of $\cK_W$, we have
isomorphisms of vector spaces:
\be
\HPV^k_c(X,W)\simeq_\C\H^k_c(\cK_W)\simeq_\C \hHPV^k(X,W)~~,
\ee
which gives the conclusion. \qed
\end{proof}

\begin{remark}
The natural isomorphism $\HPV^k_c(X,W)\simeq_\C \hHPV^k(X,W)$ also
follows directly by applying Theorem \ref{T10} to the inclusion map
$s$.  Indeed, the double complexes $(\HPV_c(X,W),\bpd,\ioda_W)$ and
$(\hHPV(X,W),\bpd,\ioda_W)$ satisfy condition A. of Section \ref{sec:spec}. The
columns of the zeroth page of the associated spectral sequence compute
the Dolbeault cohomology of the nodes of $\cK_W$, so the conditions
$\rH^{r,s}_{\dd_1}(\PV_c(X))\simeq \rH^{r,s}_{\dd_1}(\widehat{\PV}(X))$ are
satisfied for all $r,s$.  Now Theorem \ref{T10} implies:
\begin{equation}
\gr^p_\bF\,\HPV^k_c(X,W) \simeq \gr^p_\hbF\, \hHPV^k(X,W)~, ~~\forall p,k\in \Z~~.
\end{equation}
This gives the desired isomorphism $\HPV^k_c(X,W)\simeq_\C \hHPV^k(X,W)$ since 
any extension of vector spaces splits. 
\end{remark}

\subsection{Proof of Theorem A}

\

\begin{proof}
By Proposition \ref{prop:ssbulk}, the inclusion $s:\PV_c(X)\hookrightarrow
\hPV(X)$ is a morphism of complexes which induces an isomorphism
$s_\ast:\HPV_c(X,W)\stackrel{\sim}{\rightarrow} \hHPV(X,W)$. Since
$\langle \omega, \eta\rangle_B=\langle \omega, s(\eta)\rangle$ for all
$\omega\in \PV(X)$ and all $\eta\in \PV_c(X)$, we have $\langle u_1,
u_2\rangle_B^H=\langle u_1, s_\ast(u_2)\rangle^H$ for all $u_1\in \HPV(X,W)$
and all $u_2\in \HPV_c(X,W)$. Since $\langle \cdot, \cdot \rangle^H$ is
non-degenerate by Lemma \ref{lemma:bulk} and $s_\ast$ is injective, 
the pairing $\langle \cdot, \cdot
\rangle_B^H:\HPV(X,W)\times \HPV_c(X,W)\rightarrow \C$ induced by
$\langle\cdot, \cdot \rangle_B$ is also non-degenerate. On the other
hand, we have $\langle u_1, u_2\rangle_c=\langle i_\ast(u_1),u_2\rangle_B^H$ for
all $u_1,u_2\in \HPV_c(X,W)$, where $i_\ast:\HPV_c(X,W)\rightarrow
\HPV(X,W)$ is the map induced by the inclusion $i:\PV_c(X)\rightarrow
\PV(X)$ and $\langle~,~\rangle_c$ is the pairing induced by $\Tr_B$ on cohomology 
(see \cite[Proposition 5.4]{nlg1}). Since $Z_W$ is compact, the map $i_\ast$ is an isomorphism by
\cite[Proposition 3.7]{nlg1}. This shows that the pairing $\langle
\cdot, \cdot \rangle_c:\HPV_c(X,W)\times \HPV_c(X,W)\rightarrow \C$
is non-degenerate. By the results of \cite{nlg1}, we also have $\langle u_1,
u_2\rangle_c=\langle i_\ast(u_1),i_\ast(u_2)\rangle_\Omega$ for all $u_1,u_2\in
\HPV_c(X,W)$, which shows that $\langle \cdot, \cdot \rangle_\Omega$ is
non-degenerate since $i_\ast$ is bijective. \qed
\end{proof}

\subsection{Proof of Theorem A$'$}

\

\begin{proof}
The isomorphism $\HPV^k(X,W)\simeq \H^k(\cK_W)$ was proved in
\cite[Proposition 4.1]{nlg2}. Since $Z_W$ is compact, we also have an
isomorphism $\HPV_c^k(X,W)\simeq \HPV^k(X,W)$ which follows from
\cite[Proposition 3.7]{nlg1}. The isomorphism $\HPV^k(X,W)\simeq
\hHPV^{-k}(X,W)$ follows from Lemma \ref{lemma:bulk}.\qed
\end{proof}

\section{Non-degeneracy of the boundary traces}
\label{sec:boundary}

\noindent Let $\Omega$ be a holomorphic volume form on $X$. 
For any holomorphic factorization $a\!=\!(\!E\!,\!D)$ of $W\!$, 
the {\em canonical boundary trace}
$\tr_a^B\!:\!\cA_c(\!X,\!End(E))\!\rightarrow \C$ is defined by
(see \cite{nlg1}):
\be
\tr_a^B(\alpha)\eqdef \int_X \Omega\wedge \str_E(\alpha)~~,
~~\forall \alpha\in \cA_c(X,End(E))~~.
\ee
We have the canonical off-shell boundary paring: 
\be
\langle \alpha,\beta\rangle_{E_1,E_2}^B=\tr_{E_2}(\alpha\beta)~~,~
~\forall \alpha\in \cA(X,Hom(E_1,E_2))~,~~\forall\beta\in \cA_c(X,Hom(E_2,E_1))~~.
\ee

\subsection{Extended and restricted boundary pairings}

\noindent Let $E$ be a holomorphic vector superbundle on $X$. Let
$\str_E:End(E)\rightarrow \cO_X$ denote the morphism of holomorphic
vector bundles given by the fiberwise supertrace of $E$ (see
\cite[Section 4]{nlg1}). This induces a $\cinf$-linear map
$\str_E:\rOmega(X,End(E))\rightarrow \rOmega(X)$ (called the {\em
  extended supertrace}) which is determined uniquely by the condition:
\be
\str_E(\omega\otimes f)=\str_E(f) \omega~,~~\forall \omega\in \rOmega(X)~,~~\forall f\in \rGamma_\sm(X,End(E))~~.
\ee

\noindent Let $E_1$ and $E_2$ be two holomorphic vector superbundles
defined on $X$.

\

\begin{Definition}
The {\em extended boundary pairing} of $(E_1,E_2)$
is the continuous bilinear map
\be
\langle \cdot, \cdot
\rangle_{E_1,E_2}:\cA(X,Hom(E_1,E_2))\times \hcA(X,Hom(E_2,E_1))
\rightarrow \C
\ee
defined through:
\be
\langle \alpha,T\rangle_{E_1,E_2} \eqdef \int_X\Omega\wedge \str_{E_2}(\alpha T)~,~~\forall \alpha\in \cA(X,Hom(E_1,E_2))  ~,~~\forall T\in \hcA(X,Hom(E_2,E_1)) ~~.
\ee
The
restriction $\langle \cdot, \cdot
\rangle_{c,\,E_1,E_2}:\cA_c(X,Hom(E_1,E_2))\times
\cA_c(X,Hom(E_2,E_1))\rightarrow \C$ of $\langle \cdot, \cdot
\rangle_{E_1,E_2}$ is called the {\em restricted boundary pairing} of $(E_1,E_2)$.
\end{Definition}

\subsection{Relation to the Serre pairing of $Hom(E_1,E_2)$}

\noindent Let $V\eqdef Hom(E_1,E_2)$ be a $\Z_2$-graded
vector bundle whose homogeneous components are given by
$V^\0=Hom^\0(E_1,E_2)$ and $V^\1=Hom^\1(E_1,E_2)$. Since $V\simeq
E_1^\vee\otimes E_2$, we have $V^\vee\simeq E_2^\vee \otimes E_1\simeq
Hom(E_2,E_1)$. This allows us to relate the Serre pairing of
$Hom(E_1,E_2)$ to the extended boundary pairing of $(E_1,E_2)$. For
this, notice that the isomorphism $V^\vee\simeq Hom(E_2,E_1)$ can be
constructed explicitly using the morphism of holomorphic vector
bundles given by the fiberwise supertrace
$\str_{E_2}:End(E_2)\rightarrow \cO_X$.  Indeed, this induces an
isomorphism of $\Z_2$-graded holomorphic vector bundles
$\hsigma_{E_1,E_2}:Hom(E_2,E_1)\rightarrow Hom(E_1,E_2)^\vee$ which is
defined through:
\be
\hsigma_{E_1,E_2}(x)(g_x)(f_x)\eqdef \str_{E_{2,x}}(f_x\circ g_x)~~,~~\forall x\in X~,~~\forall f_x\in \Hom(E_{1,x}, E_{2,x})~,~~\forall g_x\in \Hom(E_{2,x}, E_{1,x})~.
\ee
For any $f_x\in \Hom(E_{1,x}, E_{2,x})$ and any $g_x\in \Hom(E_{2,x},E_{1,x})$, we have:
\ben
\label{evsigma}
\ev_{Hom(E_1,E_2)}(x)\big(f_x,\hsigma_{E_1,E_2}(x)(g_x)\big)=\str_{E_{2,x}}(f_x\circ g_x)~~,
\een
where $\ev$ is the graded duality pairing defined in \eqref{ev_Q}.  The even map
$\hsigma_{E_1,E_2}$ induces a topological isomorphism:
\ben
\label{tis1}
\sigma_{E_1,E_2}:\hOmega(X,Hom(E_2,E_1))\stackrel{\sim}{\longrightarrow} \hOmega\big(X,Hom(E_1,E_2)^\vee\big)
\een 
which is determined uniquely by the condition: 
\be
\sigma_{E_1,E_2}(\omega\otimes g)=\omega\otimes \hsigma_{E_1,E_2}(g)~~
\ee
for all $\omega\in \rOmega(X)$ and all $g\in
\rOmega^0(X,\!Hom(E_1,E_2))$. This isomorphism preserves the rank
bigrading as well as the bundle grading of the space
$\rOmega(X,Hom(E_1,E_2))$. Recall the map $S_Q$ defined in
\eqref{BQdef}.

\

\begin{Proposition}
\label{prop:SerreSigma} 
For any $\alpha\in \rOmega(X,Hom(E_1,E_2))$
and any $T\in \hOmega(X,Hom(E_2,E_1))$, we have:
\be
S_{Hom(E_1,E_2)}(\alpha,\sigma_{E_1,E_2}(T))=\str_{E_2}(\alpha T)~~.
\ee
In particular, the Serre pairing of the $\Z_2$-graded vector bundle
$Hom(E_1,E_2)$ satisfies:
\be
\S_{Hom(E_1,E_2)}(\alpha,\sigma_{E_1,E_2}(T))=\int_X \str_{E_2}(\alpha T)~~.
\ee
\end{Proposition}

\begin{proof} 
Follows immediately from \eqref{evsigma}, \eqref{cBQdef}
and from the definition of the extended supertrace.\qed
\end{proof}

\noindent The wedge product with $\Omega$ induces a topological isomorphism: 
\ben
\label{tis2}
\Omega\wedge \,: \hcA(X, Hom(E_1,E_2))\stackrel{\sim}{\longrightarrow} \hOmega^{d,\bullet}(X, Hom(E_1,E_2))~~.
\een

\vspace{3mm} 

\begin{Corollary}
For any $\alpha\in \cA(X,Hom(E_1,E_2))$ and any $T\in \hcA(X,Hom(E_2,E_1))$, we have: 
\be
S_{Hom(E_1,E_2)}(\Omega\wedge \alpha,\sigma_{E_1,E_2}(T))=\Omega\wedge \str_{E_2}(\alpha T)~~.
\ee
In particular, the Serre pairing of the $\Z_2$-graded vector bundle $Hom(E_1,E_2)$ satisfies: 
\be
\S_{Hom(E_1,E_2)}(\Omega\wedge \alpha,\sigma_{E_1,E_2}(T))=\int_X \Omega\wedge \str_{E_2}(\alpha T)=\langle \alpha, T\rangle_{E_1,E_2}~~.
\ee
\end{Corollary}

\begin{proof}
Follows immediately from Proposition \ref{prop:SerreSigma} and from
the definition of the extended supertrace.  \qed
\end{proof}

\begin{Proposition}
\label{prop:bdry_perf}
The extended boundary pairing $\langle \cdot, \cdot \rangle_{E_1,E_2}$
is a perfect topological pairing between the topological vector spaces
$\cA(X,Hom(E_1,E_2))$ and $\hcA(X,Hom(E_2,E_1))$.
\end{Proposition}

\begin{proof}
Follows from Proposition \ref{prop:SerreSigma} and Lemma \ref{lemma:Serre},
using the fact that \eqref{tis1} and \eqref{tis2} are topological
isomorphisms. \qed
\end{proof}

\subsection{Holomorphic factorizations}

Recall that a {\em holomorphic factorization} of $W$ is a pair
$(E,D)$, where $E$ is a holomorphic vector superbundle and $D\in
\rGamma(X,End^\1(E))$ is an odd holomorphic section of $E$ such that
$D^2=W\id_E$ (see \cite{nlg1}). As explained in op. cit., holomorphic
factorizations of $W$ form an $\O(X)$-linear and $\Z_2$-graded dg-category 
$\DF(X,W)$ known as the {\em twisted Dolbeault category of
holomorphic factorizations}, whose morphism spaces are the spaces
\eqref{HomDF} of bundle-valued forms, endowed with the twisted
Dolbeault differentials \eqref{tdolbeault}. 

\subsection{The topological complexes $\big(\cA(X,Hom(E_1,E_2)), \updelta_{a_1,a_2}\big)$ 
and $\big(\cA_c(X,Hom(E_1,E_2)), \updelta_{a_1,a_2}\big)$ }

Let $a_1=(E_1,D_1)$ and $a_2=(E_2,D_2)$ be two holomorphic
factorizations of $W$. We endow the space $\cA(X,Hom(E_1,E_2))$ with
the twisted Dolbeault differential $\updelta:=\updelta_{a_1,a_2}\eqdef
\bpd+\fd$, where $\bpd:=\bpd_{Hom(E_1,E_2)}$ and $\fd:=\fd_{a_1,a_2}$ is the defect
differential (see \eqref{tdolbeault} and \eqref{defect}). 

\

\begin{Proposition}
$\big(\cA(X,Hom(E_1,E_2)), \updelta_{a_1,a_2}\big)$ is a $\Z_2$-graded
  topological complex of FS spaces. 
\end{Proposition}

\begin{proof}
It is clear that $\cA(X,Hom(E_1,E_2))$ is an FS space
while $\bpd_{Hom(E_1,E_2)}$ is a continuous differential. The
conclusion now follows by noticing that $\fd_{a_1,a_2}$ is continuous.\qed
\end{proof}

\

\noindent  Notice that $\cA_c(X,Hom(E_1,E_2))$ is a closed subspace of
$\cA(X,Hom(E_1,E_2))$ as well as a subcomplex of
$(\cA(X,Hom(E_1,E_2)), \updelta_{a_1,a_2})$. Let
$\HDF_{X,W}(a_1,a_2)$ denote the cohomology of the topological complex
$(\cA(X,Hom(E_1,E_2)),\updelta_{a_1,a_2})$ and $\HDF_{\!c,X,W}(a_1,a_2)$
denote the cohomology of $(\cA_c(X,Hom(E_1,E_2)),\updelta_{a_1,a_2})$.

\subsection{The topological complex $\big(\hcA(X\!,\!Hom(\!E_1,\!E_2)\!), \!\hupdelta_{a_1,a_2}\!\big)$}
Let $\hcA(\!X\!,\!Hom(\!E_1,\!E_2\!)\!)\!=\!\hOmega^{0,\bullet}\!(\!X\!,\!Hom(\!E_1,\!E_2)\!)$
denote the DFS space of $Hom(E_1,E_2)$-valued compactly-supported
currents of type $(0,\bullet)$ defined on $X$, which contains
$\cA_c(X,Hom(E_1,E_2))$ as a closed subspace. Let $\hupdelta_{a_1,a_2}$
be the canonical extension of $\updelta_{a_1,a_2}$ to $\hcA(X,
Hom(E_1,E_2))$. 

\vspace{3mm}

\begin{Proposition}
$\big(\hcA(X,Hom(E_1,E_2)), \hupdelta_{a_1,a_2}\big)$ is a $\Z_2$-graded
  topological complex of DFS spaces.
\end{Proposition}

\begin{proof}
The fact that $\hcA(X,Hom(E_1,E_2))$ is a DFS space follows from
Proposition \eqref{prop:bdry_perf}, while continuity of $\hupdelta_{a_1,a_2}$
follows from the fact that $\updelta_{a_1,a_2}$ is continuous. \qed
\end{proof}

\

\noindent Let $\hHDF_{X,W}(a_1,a_2)$ denote the cohomology of the complex
$\big(\hcA(X,Hom(E_1,E_2)),\hupdelta_{a_1,a_2}\big)$. 

\

\begin{Proposition}
\label{prop:bdry_top}
The extended boundary pairing is a perfect topological pairing of
complexes between $\cA(X,Hom(E_1,E_2))$ and $\hcA(X,\Hom(E_1,E_2))$.
\end{Proposition}

\begin{proof}
The fact that $\langle \cdot, \cdot \rangle_{E_1,E_2}$ is a pairing of
complexes follows by direct computation as in \cite{nlg1}. The fact
that this is a perfect pairing follows from Proposition
\ref{prop:bdry_perf}. \qed
\end{proof}

\subsection{Non-degeneracy of the cohomological boundary pairings}

\

\

\begin{Lemma}
\label{lemma:bdry}
Suppose that the critical set $Z_W$ is compact. Then the vector spaces
$\HDF_{X,W}(a_1,a_2)$ and $\hHDF_{X,W}(a_1,a_2)$ are
  finite-dimensional and the extended canonical off-shell boundary pairing
  $\langle \cdot, \cdot \rangle_{E_1,E_2}$ is cohomologically perfect. In
  particular, we have induced isomorphisms of $\Z_2$-graded vector
  spaces:
\be
\hHDF_{X,W}(a_1,a_2)\simeq_\C \HDF_{X,W}(a_1,a_2)^\vee~~.
\ee
\end{Lemma}

\begin{proof}
Since $Z_W$ is compact, the vector space
$\HDF_{X,W}(a_1,a_2)=\rH(\cA(X,Hom(E_1,\!E_2)),\updelta_{a_1,a_2})$ is
finite-dimensional, as shown in \cite{nlg1}. The remaining statements
follow from Proposition \ref{prop:fdDual}. \qed
\end{proof}

\vspace{3mm}

\noindent Let $s\!:\!\cA_c(X,Hom(E_1,\!E_2))\!\rightarrow\!
\hcA(X,Hom(E_1,\!E_2))$ be the inclusion map 
and $s_\ast\!:\!\HDF_{\!c,X,W}(a_1,a_2)\!\rightarrow\! \hHDF_{X,W}(a_1,a_2)$ be
the linear map induced by $s$ on cohomology.

\

\begin{Proposition}
\label{prop:ssbdry}
Suppose that $Z_W$ is compact. Then $\HDF_{X,W}(a_1,a_2)$ and
$\hHDF_{X,W}(a_1,a_2)$ are finite-dimensional and the map $s_\ast$ is an
isomorphism of vector spaces.
\end{Proposition}

\begin{proof}
As in \cite[Subsection 5.1]{nlg2}, we define a horizontally 2-periodic
double complex $K=\oplus_{i,j} K^{i,j}$ which is an unwinding of
$\cA_c(X,Hom(E_1,E_2))$:
\be
K^{i,j} \eqdef \cA^j(X, Hom^{{\hat i}}(E_1,E_2))~,~~\forall i, j\in \Z~~,
\ee
with vertical differentials given by $\bbpd:=\bbpd_{a_1,a_2}$ and horizontal
differentials given by $(-1)^j\fd$. Then the total complex $\oplus_{n} K^n=
\oplus_{i+j=n} K^{i,j}$ has the differential
$\updelta_{a_1,a_2}=\bbpd+\fd$ and is 2-periodic. We similarly define
a double complex $\hK$ as the unwinding of $\hcA(X,Hom(E_1,E_2))$.  The
inclusion of single complexes $s:\cA_c(X,Hom(E_1,E_2))\hookrightarrow
\hcA(X,Hom(E_1,E_2))$ naturally defines an inclusion of double
complexes $s:(K,\bbpd,(-1)^j\fd)\hookrightarrow
(\hK,\widehat{\bbpd},(-1)^j\widehat{\fd})$.  The standard filtration
$F$ on the double complexes defined as in (\ref{f9}) is 2-periodic and induces
filtrations $\bF$ and $\hbF$ on $K$ and $\hK$ respectively.  We are
going to use Theorem \ref{T10} for the inclusion $s$.  Consider the
2-periodic complex of locally-free sheaves:
\ben
\label{DefComplex}
(\cE_{a_1,a_2}):~ ~\ldots \longrightarrow Hom^\1(E_1, E_2)\stackrel{\md}{\longrightarrow}
 Hom^\0(E_1, E_2) \stackrel{\md}{\longrightarrow} Hom^\1(E_1,E_2)\longrightarrow \ldots~~
\een
with $Hom^\0(E_1, E_2) $ sitting in even degrees.  The columns of both
double complexes form Dolbeault resolutions of the nodes of
$\cE_{a_1,a_2}$. Thus, the zero pages of the associated spectral
sequences coincide. In our case this reads:
\be
\bE_0^{p,q}=\rH^{p,q}_{\bbpd}(K)=\rH^{p,q}_{\hbbpd}(\hK)=\bhE_0^{p,q}~,~~ \forall p,q\in \Z~~.
\ee
Since both induced filtrations on total complexes satisfy condition
B. of Section \ref{sec:spec}, the associated spectral sequences
converge to the total cohomologies $H^n_\updelta(K)$ and
$H^n_{\hupdelta}(\hK)$ by Proposition \ref{pro8}.  Now Theorem
\ref{T10} implies that for the graded pieces of the induced filtration
on the total cohomology, the following isomorphism holds for every
$p$:
\begin{equation}
\label{grKhKiso}
\gr_\bF^p\rH^n_\updelta(K)\simeq \gr_\bF^p\rH^n_{\hupdelta}(\hK)~,~~\forall n,p\in \Z ~~.
\end{equation}
Returning to $\HDF_{\!c,X,W}(a_1,a_2)$, we can express this in terms
of the cohomology of the 2-periodic total complex:
\be
\rH_{\updelta}^n(K)=\rH^{{\hat n}}(\cA(X,Hom(E_1,E_2)),\updelta)=\Hom_{\HDF_{c,X,W}}^{\hat n}(a_1,a_2)~,~~\forall n\in \Z
\ee
and similarly for the complex $\hK$. This give an isomorphism:
\begin{equation}
 \gr_\hbF^p\hHDF_{X,W}^i(a_1,a_2) \simeq	\gr_\bF^p\HDF^i_{\!c,X,W}(a_1,a_2)~,~~\forall p ~~\text{for } i=0,1.
\end{equation}
Since the groups $\Ext^i(\cdot,\cdot)$ vanish for $i>0$ in the category
of vector spaces over $\C$, we obtain the desired isomorphism 
$s_\ast:\hHDF_{X,W}(a_1,a_2)\stackrel{\sim}{\rightarrow} \HDF_{\!c,X,W}(a_1,a_2)$.
\qed
\end{proof}

\

\begin{Proposition}
\label{prop:bdry}
Suppose that $Z_W$ is compact. Then the restricted boundary pairing
\be
\langle \cdot, \cdot \rangle_{c,\,E_1,E_2}:\cA_c(X,Hom(E_1,E_2))\times
\cA_c(X,Hom(E_2,E_1))\longrightarrow \C
\ee
is cohomologically non-degenerate.
\end{Proposition}

\begin{proof}
Consider the inclusion $j:\cA_c(X, Hom(E_1,E_2))\hookrightarrow
\cA(X,Hom(E_1,E_2))$.  Then it was shown in \cite{nlg1} that $j$ is a
quasi-isomorphism of complexes from $(\cA_c(X, Hom(E_1,E_2)),
\updelta_{a_1,a_2})$ to $(\cA(X, Hom(E_1,E_2)), \updelta_{a_1,a_2})$.
Let $j_\ast:\HDF_{\!c,X,W}(a_1,a_2)\stackrel{\sim}{\rightarrow}
\HDF_{X,W}(a_1,a_2)$ be the isomorphism induced by this map on
cohomology. By Proposition \ref{prop:ssbdry}, the inclusion
$s:\cA_c(X,Hom(E_1,E_2))\hookrightarrow
\hcA(X,Hom(E_1,E_2))$ proves to be also a quasi-isomorphism since it
induces an isomorphism
$s_\ast:\HDF_{c,X,W}(a_1,a_2)\!\stackrel{\sim}{\rightarrow}
\hHDF_{X,W}(a_1,a_2)$.

For every $\alpha\!\in\! \cA_c(X,Hom(E_1,E_2))$ and every $\beta\! \in\!
\cA_c(X,Hom(E_2,E_1))$, we have $\langle \alpha,
\beta\rangle_{c,\,E_1,E_2}\!=\!\langle j(\alpha),
s(\beta)\rangle_{E_1,E_2}$. On the cohomological level, this implies
$\langle u,v\rangle_{c,\,H,a_1,a_2}=\langle
j_\ast(u),s_\ast(v)\rangle_{H,a_1,a_2}$ for all $u\in
\HDF_{\!c,X,W}(a_1,a_2)$ and all $v\in \HDF_{\!c,X,W}(a_2,a_1)$.  On
the other hand, $\HDF_{\!c,X,W}(a_1,a_2)$ and
$\HDF_{\!c,X,W}(a_2,a_1)$ are finite-dimensional and $\langle \cdot,
\cdot \rangle_{H,a_1,a_2}$ is non-degenerate by Lemma
\ref{lemma:bdry}. This shows that $\langle \cdot, \cdot
\rangle_{c,H,a_1,a_2}$ is also non-degenerate since $j_\ast$ and
$s_\ast$ are bijective.  \qed
\end{proof}

\subsection{Proof of Theorem B}

\

\begin{proof}
By Proposition \ref{prop:ssbdry}, the inclusion map
$s\!:\!\cA_c(X,Hom(E_1,\!E_2))\!\rightarrow\!  \hcA(X,Hom(E_1,\!E_2))$
induces an isomorphism
$s_\ast\!:\!\HDF_{\!c,X,W}(a_1,a_2)\!\rightarrow\!
\hHDF_{X,W}(a_1,a_2)$. Since $\langle \alpha_1,
\alpha_2\rangle^B_{E_1,E_2}\!\!=\!\langle \alpha_1,
s(\alpha_2)\rangle_{E_1,E_2}$ for all $\alpha_1\in
\cA(X,Hom(E_1,E_2))$ and all $\alpha_2\in \cA_c(X,Hom(E_2,E_1))$, we
have $\langle t_1, t_2\rangle^{BH}_{E_1,E_2}\!=\!\langle t_1,
s_\ast(t_2)\rangle_{E_1,E_2}^H$ for all $t_1\in \HDF_{X,W}(a_1,a_2)$
and all $t_2\in \HDF_{c,X,W}(a_2,a_1)$. Since $\langle \cdot, \cdot
\rangle_{E_1,E_2}^H$ is non-degenerate by Lemma \ref{lemma:bdry} and
$s_\ast$ is injective, it follows that the pairing $\langle \cdot,
\cdot \rangle^{BH}_{a_1,a_2}:\HDF_{X,W}(a_1,a_2)\times
\HDF_{c,X,W}(a_2,a_1)\rightarrow \C$ induced by $\langle\cdot, \cdot
\rangle^B_{E_1,E_2}$ is also non-degenerate.  On the other hand, we
have $\langle t_1,t_2\rangle^c_{a_1,a_2}=\langle
j_\ast(t_1),t_2\rangle_{E_1,E_2}^{BH}$ for all $t_1\in
\HDF_{c,X,W}(a_1,a_2)$ and all $t_2\in \HDF_{c,X,W}(a_2,a_1)$, where
$j_\ast:\HDF_{c,X,W}(a_1,a_2)\rightarrow \HDF_{X,W}(a_1,a_2)$ is the
map induced by the inclusion $j:\cA_c(X,Hom(E_1,E_2))\rightarrow
\cA(X,Hom(E_1,E_2))$ and $\langle \cdot , \cdot \rangle^c_{a_1,a_2}$
is the pairing induced by $\tr_{a_2}^B$ on cohomology (see
\cite[Proposition 6.3]{nlg1}). Since $Z_W$ is compact, the map
$j_\ast$ is an isomorphism by \cite[Proposition 4.11]{nlg1}. This
shows that the pairing $\langle \cdot, \cdot
\rangle_{a_1,a_2}^c:\HDF_{c,X,W}(a_1,a_2)\times
\HDF_{c,X,W}(a_2,a_1)\rightarrow \C$ is non-degenerate. By the results
of \cite{nlg1}, we also have $\langle t_1,
t_2\rangle^c_{a_1,a_2}=\langle
j_\ast(t_1),j_\ast(t_2)\rangle^\Omega_{a_1,a_2}$ for all $t_1\in
\HDF_{c,X,W}(a_1,a_2)$ and all $t_2\in \HDF_{c,X,W}(a_2,a_1)$, which
shows that $\langle \cdot, \cdot \rangle_{a_1,a_2}^\Omega$ is
non-degenerate since $j_\ast$ is bijective. \qed
\end{proof}

As mentioned in the introduction, Theorem \ref{thm:B} can be reformulated using Serre 
functors. For this, we first discuss the shift functor of the category $\HDF(X,W)$.
We refer the reader to Appendix \ref{app:shift} for some notions and properties 
used in the next subsections. 

\subsection{The shift functor of the category $\VB_\sm(X)$}

Let $\VB_\sm(X)$ be the $\Z_2$-graded $\cinf$-linear category of
smooth vector bundles and smooth sections defined in \cite{nlg1}.
This category admits a shift functor $\rPi$ defined as follows:
\begin{enumerate}
\itemsep 0.0em
\item For any $\Z_2$-graded vector bundle $E$ defined on $X$, let
  $\rPi(E)$ denote the $\Z_2$-graded vector bundle with homogeneous
  components:
\be
\rPi (E)^\0\eqdef E^\1~~,~~\rPi (E)^\1\eqdef E^\0~.
\ee
\item For any morphism $s:E_1\rightarrow E_2$ in $\VB_\sm(X)$
  (i.e. for any smooth section $s\in \rGamma_\sm(X, Hom(E_1,E_2))$
  of the vector bundle $Hom(E_1,E_2)=E_1^\vee\otimes E_2$, let:
\be
\rPi(s)\eqdef s~~,
\ee 
where the right hand side is viewed as a section of the bundle $Hom(\rPi(E_1),\rPi(E_2))$.
\end{enumerate}

\begin{remark} 
For any $\Z_2$-graded vector bundle $E$, let $\sigma_E:E \rightarrow
\rPi(E)$ be the {\em suspension morphism} of $E$, i.e. the identity
endomorphism of $E$ viewed as an odd morphism for $\Z_2$-graded vector
bundles from $E$ to $\rPi(E)$. This can be viewed as a section of the
$\Z_2$-graded vector bundle $Hom(E,E)$ and hence as on odd invertible
morphism from $E$ to $\rPi(E)$ in the category $\VB_\sm(X)$.  The
fact that $\rPi$ is a functor follows by noticing the relation:
\be
\rPi(s)=\sigma_{E_2} \circ s\circ \sigma_{E_1}^{-1}~~,~~\forall s\in \rGamma_\sm(X,Hom(E_1,E_2))~~,
\ee
where $\circ$ denotes the composition of $\VB_\sm(X)$. 
Writing $s$ as a block matrix $s=\left[\begin{array}{cc}
    s_{00} & s_{01}\\s_{10} & s_{11} \end{array} \right]$ (where
$s_{\kappa_1\kappa_2}\in \rGamma_\sm(X, Hom(E_1^{\kappa_2},
E_2^{\kappa_1}))$ for all $\kappa_1, \kappa_2\in \Z_2$), we have: 
\be
\rPi(s)_{\kappa_1\kappa_2}=s_{\kappa_1+\1,\kappa_2+\1}~~,
\ee
i.e.:
\be
\rPi(s)=\left[\begin{array}{cc}
    s_{11} & s_{10}\\s_{01} & s_{00} \end{array} \right]~~,
\ee
and the suspension morphism of $E$ corresponds to the matrix:
\be
\sigma_E=\left[\begin{array}{cc}
    0 & \id_{E^\1} \\ \id_{E^\0}  & 0 \end{array} \right]~~.
\ee
\end{remark}

\subsection{The shift functors of $\DF(X,W)$ and $\HDF(X,W)$}

Let $\DF(X,W)$ be the twisted Dolbeault category of the holomorphic LG
pair $(X,W)$ (which is a $\Z_2$-graded $\O(X)$-linear category). Let
$\Sigma$ be the automorphism of the underlying $\Z_2$-graded
$\O(X)$-linear category which is defined as follows:

\begin{enumerate}
\itemsep 0.0em
\item For any holomorphic factorization $(E,D)$ of $W$, let:
\be
\Sigma(E,D)\eqdef(\rPi E,\rPi D)~.
\ee
\item Given two holomorphic factorizations $a_1=(E_1,D_1)$ and
  $a_2=(E_2,D_2)$ of $W$ and a morphism $\alpha \in
  \Hom_{\DF(X,W)}(a_1,a_2)=\cA(X,Hom(E_1,E_2))=\cA(X)\otimes_{\cinf}\rGamma_\sm(X,Hom(E_1,E_2))$,
  let:
\be
\Sigma(\alpha)\eqdef (\id_{\cA(X)}\otimes \rPi)(\alpha)~~.
\ee
\end{enumerate}

\

\begin{remark}
Writing $D=\left[\begin{array}{cc} 0 & G\\ F & 0 \end{array}\right]$
(with $F\in \rGamma(X,Hom(E^\0,E^\1))$ and $G\in
\rGamma(X,Hom(E^\1,E^\0))$), we have $\rPi D\eqdef
\left[\begin{array}{cc} 0 & F\\ G & 0 \end{array}\right]$~.
\end{remark}

\

\begin{Proposition}
$\Sigma$ is a differential shift functor for the $\Z_2$-graded
  $\O(X)$-linear dg-category $\DF(X,W)$.
\end{Proposition}

\

\begin{proof}
A simple computation shows that condition \eqref{dSigma} is
satisfied. \qed
\end{proof}

\

\noindent Since $\Sigma$ is a differential shift functor on $\DF(X,W)$, it induces a shift 
functor on the $\Z_2$-graded $\O(X)$-linear category $\HDF(X,W)=\rH(\DF(X,W))$, which we again 
denote by $\Sigma$. 

\

\begin{Proposition}
\label{prop:cypc}
Suppose that $Z_W$ is compact. Then $(\HDF(X,W), \tr, \Sigma)$ is a
Calabi-Yau supercategory of parity $\mu=\hat d\in\Z_2$ with compatible shift functor 
in the sense of 
Definition \ref{def:compat} (see Appendix \ref{app:shift}).
\end{Proposition}

\

\begin{proof}
Follows immediately from Theorem \ref{thm:B} and the definition of $\Sigma$. \qed
\end{proof}

\subsection{Proof of Theorem B$\,'$}

\

\

\begin{proof}
Follows immediately from Proposition \ref{prop:cypc} and
Proposition \ref{prop:SerreFunctor} of Appendix \ref{app:shift}. \qed
\end{proof}

\appendix

\section{Linear categories and supercategories with involutive shift functor}

\label{app:shift}

\noindent In this Appendix, we collect some facts regarding linear
categories with involutive shift functors and Serre functors. We
are particularly interested in the case of $\mu$-Calabi-Yau categories in
the sense of \cite{nlg1}. For simplicity, we assume shift functors
and Serre functors to be automorphisms (rather than autoequivalences),
since this case suffices for the purpose of the present paper.

\subsection{The $\Z_2$-graded category $\Mod_R^s$ and its shift functor}

Let $R$ be a unital commutative ring and $\Mod_R$ be the category of
$R$-modules.  Recall that an {\em $R$-supermodule} is a $\Z_2$-graded
$R$-module, i.e. an $R$-module $M$ endowed with a direct sum
decomposition $M=M^\0\oplus M^\1$ into two submodules $M^\0$ and
$M^\1$.  Let $\Mod_R^{\Z_2}$ denote the ordinary category of
$\Z_2$-graded $R$-modules, whose set of morphisms from an
$R$-supermodule $M$ to an $R$-supermodule $N$ is the (ungraded)
$R$-module:
\be
\Hom(M,N)\eqdef \Hom_R(M^\0,N^\0)\oplus \Hom_R(M^\1,N^\1)~~.
\ee
Let $\Mod_R^s$ be the category whose objects are the $R$-supermodules
and whose set of morphisms from an object $M$ to an object $N$ is the
{\em inner Hom $R$-supermodule}\, $\uHom(M,N)$, whose homogeneous
components are defined through:
\beqa
&&\uHom^\0(M,N)\eqdef \Hom_R(M^\0,N^\0)\oplus \Hom_R(M^\1,N^\1)~,\\
&&\uHom^\1(M,N)\eqdef \Hom_R(M^\0,N^\1)\oplus \Hom_R(M^\1,N^\0)~~
\eeqa
and whose composition of morphisms is induced from $\Mod_R$. We have $\uHom^\0(M,N)=\Hom(M,N)$. 

\

\begin{Definition}
The {\em parity change functor $\rPi$} of $\Mod_R^s$ is the
automorphism of $\Mod_R^s$ defined as follows:
\begin{enumerate}
\itemsep 0.0em
\item For any $R$-supermodule $M=M^\0\oplus M^\1$, the $R$-supermodule
  $\rPi(M)$ has homogeneous components:
\be
\rPi(M)^\0\eqdef M^\1~,~~\rPi(M)^\1\eqdef M^\0~~.
\ee 
\item For any morphism $f\in \uHom(M,N)$ of $\Mod_R^s$, the morphism
  $\rPi(f)\in \uHom(\rPi(M),\rPi(N))$ has homogeneous components:
\beqa
&& \rPi(f)^\0\eqdef f^\0\in \Hom_R(M^\0,N^\0)\oplus \Hom_R(M^\1,N^\1)=\uHom^\0(\rPi(M),\rPi(N))~~,\\
&& \rPi(f)^\1\eqdef f^\1\in \Hom_R(M^\0,N^\1)\oplus \Hom_R(M^\1,N^\0)=\uHom^\1(\rPi(M),\rPi(N))~~.
\eeqa
\end{enumerate}
\end{Definition}

\

\noindent It is clear that $\rPi$ is involutive, i.e. we have
$\rPi^2=\id_{\Mod_R^s}$, where $\id_{\Mod_R^s}$ denotes the identity
functor of $\Mod_R^s$. For any $R$-supermodules $M$ and $N$, we have:
\be
\uHom(M,\rPi(N))\simeq \rPi\uHom(M,N)=\uHom(\rPi(M),N)~~,
\ee
where the second equality results from the first upon replacing $M$
and $N$ with $\rPi(M)$ and $\rPi(N)$ respectively.

\subsection{Shift functors on linear supercategories}

Let $\cT$ be an {\em $R$-linear supercategory}, i.e. a category enriched over $\Mod_R^{\Z_2}$. 
A linear functor $F:\cT\rightarrow \cT$ is called {\em even} if the following 
condition holds for any two objects $a$ and $b$ of $\cT$:
\be
F(\Hom_\cT^\kappa(a,b))\subset \Hom_\cT^\kappa(a,b)~~,~~\forall \kappa\in \Z_2~~.
\ee
The {\em even subcategory} $\cT$ is the subcategory
obtained from $\cT$ taking the same objects but keeping only those
morphisms which have degree $\0\in \Z_2$ (without changing the
composition of morphisms). We denote this subcategory by $\Ev(\cT)$ or by $\cT^\0$. 

\

\begin{Definition}
A {\em shift functor} for $\cT$ is an even automorphism $\Sigma$
of $\cT$ which satisfies the following properties:
\begin{enumerate}[1.]
\itemsep 0.0em
\item We have $\Sigma^2=\id_\cT$ .
\item For any two objects $a$ and $b$ of $\cT$, there exist
  isomorphisms of $\Z_2$-graded $R$-modules:
\ben
\label{rho}
\Hom_\cT(a,\Sigma(b))\stackrel{\rho_{ab}}{\longrightarrow} \rPi\Hom_\cT(a,b) 
\een 
which are natural in both $a$ and $b$. More precisely, there exists an isomorphism: 
\be
\rho:\Hom_\cT\circ(\id_\cT\times \Sigma)\stackrel{\sim}{\rightarrow} \rPi\circ \Hom_{\cT}~~
\ee
in the category of functors from $\cT\times \cT$ to $\cT$ and natural
transformations between such.
\end{enumerate}
In this case, the pair $(\cT,\Sigma)$ is called an {\em $R$-linear supercategory with shift}.
\end{Definition}

\

\begin{remark} 
Let $(\cT,\Sigma)$ be an $R$-linear supercategory with
shift. The replacement in \eqref{rho} of $a$ and $b$ by $\Sigma(a)$ and
$\Sigma(b)$ respectively gives isomorphisms:
\be
\Hom_\cT(\Sigma(a),b)=\Hom_\cT(\Sigma(a),\Sigma^2(b))\xrightarrow{\rho_{\Sigma(a)\Sigma(b)}}\rPi\Hom_\cT(\Sigma(a),\Sigma(b)) \stackrel{\Sigma}{\longrightarrow} \rPi\Hom_\cT(a,b)~~,
\ee
where we used the relation $\Sigma^2=\id_\cT$. We thus have a composite isomorphism: 
\be
\Hom_\cT(\Sigma(a),b) \xrightarrow{\Sigma \circ \rho_{\Sigma(a)\Sigma(b)}}\rPi\Hom_\cT(a,b)~~,
\ee
which is natural in both $a$ and $b$. 
\end{remark}

\

\begin{Definition}
Let $(\cT_1,\Sigma_1)$ and $(\cT_2,\Sigma_2)$ be two $R$-linear
supercategories with shifts. A {\em morphism of $R$-linear
  supercategories with shifts} from $(\cT_1,\Sigma_1)$ to
$(\cT_2,\Sigma_2)$ is a linear functor $F:\cT_1\rightarrow \cT_2$ such that
$F\circ \Sigma_1=\Sigma_2\circ F$.
\end{Definition}

\

\noindent With this definition of morphisms, $R$-linear
supercategories with shifts form a category denoted
$\mathrm{RSCat}^s$.

\subsection{$R$-linear categories with involution} 
Let $\cC$ be an $R$-linear category, i.e. a category enriched over
$\Mod_R$.

\

\begin{Definition}
An {\em involution} of $\cC$ is a linear automorphism
$\Sigma$ of $\cC$ such that $\Sigma^2=\id_\cC$. In this case, the pair 
$(\cC,\Sigma)$ is called an {\em $R$-linear category with involution}. 
\end{Definition}

\

\begin{Definition}
Let $(\cC_1,\Sigma_1)$ and $(\cC_2,\Sigma_2)$ be two $R$-linear categories 
with involution. A {\em morphism of $R$-linear categories with involution} 
from $(\cC_1,\Sigma_1)$ to $(\cC_2,\Sigma_2)$ is a linear functor $F:\cC_1\rightarrow \cC_2$ 
such that $F\circ \Sigma_1=\Sigma_2\circ F$. 
\end{Definition}

\

\noindent With this definition, $R$-linear categories with involution form a category denoted $\mathrm{RICat}$. 

\subsection{Supercompletion of an $R$-linear category with involution}

An $R$-linear category with involution can be completed to a $\Z_2$-graded category as follows.

\

\begin{Definition}
Let $(\cC,\Sigma)$ be an $R$-linear category with involution. The {\em
  supercompletion} of $\cC$ along $\Sigma$ is the $R$-linear
$\Z_2$-graded category $\Gr_\Sigma(\cC)$ defined as follows:
\begin{enumerate}
\itemsep 0.0em
\item The objects of $\Gr_\Sigma(\cC)$ coincide with those of $\cC$.
\item For any objects $a,b$ of $\cC$ and any $\kappa\in \Z_2$, the
  $R$-module of morphisms from $a$ to $b$ in $\cC$ has the
  $\Z_2$-grading given by the decomposition
  $\Hom_{\Gr_\Sigma(\cC)}(a,b)=\Hom_{\Gr_\Sigma(\cC)}^\0(a,b)\oplus
  \Hom^\1_{\Gr_\Sigma(\cC)}(a,b)$, where:
\be
\Hom_{\Gr_\Sigma(\cC)}^\kappa(a,b) \eqdef \Hom_{\cC}(a,\Sigma^\kappa(b))~,~\forall \kappa\in \Z_2~~.
\ee
\item Given three objects $a,b,c$ of $\cC$, the $R$-bilinear composition of morphisms 
$\circ:\Hom_{\Gr_\Sigma(\cC)}(b,c)\times \Hom_{\Gr_\Sigma(\cC)}(a,b)\rightarrow \Hom_{\Gr_\Sigma(\cC)}(a,c)$ 
of $\Gr_\Sigma(\cC)$ is uniquely determined by the condition: 
\be
g\circ_{\Gr_\Sigma(\cC)} f\eqdef \Sigma^\kappa(g)\circ f\in \Hom_{\cC}(a,\Sigma^{\kappa+\nu}(c))=\Hom_{\Gr_\Sigma(\cC)}^{\kappa+\nu}(a,c)~~,
\ee
for $f\in \Hom_{\Gr_\Sigma(\cC)}^{\kappa}(a,b)=
\Hom_{\cC}(a,\Sigma^\kappa(b))$ and $g\in
\Hom_{\Gr_\Sigma(\cC)}^{\nu}(b,c)= \Hom_{\cC}(b,\Sigma^\nu(c))$ (where
$\kappa,\nu\in \Z_2$).
\end{enumerate}
\end{Definition}

\

\noindent The proof of the following statements is elementary and left to the reader:

\

\begin{Proposition}
Let $(\cC,\Sigma)$ be an $R$-linear category with involution.
Consider the functor $\Gr(\Sigma):\Gr_\Sigma(\cC)\rightarrow
\Gr_\Sigma(\cC)$ defined as follows:
\begin{enumerate}
\itemsep 0.0em
\item For any object $a$ of $\cC$, let $\Gr(\Sigma)(a)\eqdef
  \Sigma(a)$.
\item For any morphism $f=u\oplus v\in
  \Hom_{\Gr(\cC)}(a,b)=\Hom_{\cC}(a,b)\oplus \Hom_{\cC}(a,\Sigma(b))$
  (where $u\in \Hom_{\cC}(a,b)$ and $v\in \Hom_{\cC}(a,\Sigma(b))$),
  let: 
\be
\Gr(\Sigma)(f)\eqdef \Sigma(u)\oplus \Sigma(v) \in
  \Hom_{\Gr_\Sigma(\cC)}(\Sigma(a),\Sigma(b))=\Hom_{\cC}(\Sigma(a),\Sigma(b))\oplus \Hom_{\cC}(\Sigma(a),b)~~.
\ee
\end{enumerate}
Then $\Gr(\Sigma)$ is a shift functor for the supercompletion $\Gr_\Sigma(\cC)$.
\end{Proposition}

\

\begin{Proposition}
Let $F:(\cC_1,\Sigma_1)\rightarrow (\cC_2,\Sigma_2)$ be a morphism of
$R$-linear categories with involution. Consider the functor
$\Gr(F):\Gr_{\Sigma_1}(\cC_1)\rightarrow \Gr_{\Sigma_2}(\cC_2)$
defined through:
\begin{enumerate}
\itemsep 0.0em
\item For any object $a$ of $\cC_1$, set $\Gr(F)(a)\eqdef F(a)$.
\item For any morphism $f\!=\!u\oplus v\in
 \! \Hom_{\Gr_{\Sigma_1}\!(\cC_1)}\!(a,b)$, where $u\in\!\Hom_{\cC_1}\!(a,b)$ and
  $v\in \!\Hom_{\cC_1}(a,\Sigma_1(b))$, set:
\be
\Gr(F)(f)\!\eqdef \!F(u) \oplus F(v) \in\! \Hom_{\Gr_{\Sigma_2}(\cC_2)}(\!F(a),\!F(b))\!=\!\Hom_{\cC_2}(\!F(a),\!F(b))\oplus \Hom_{\cC_2}(\!F(a),\!\Sigma_2(\!F(b)))~~,
\ee
where we used the relation $F\circ \Sigma_1=\Sigma_2\circ F$.
\end{enumerate}
Then $\Gr(F):(\Gr_{\Sigma_1}(\cC_1),\Gr(\Sigma_1))\rightarrow
(\Gr_{\Sigma_2}(\cC_2),\Gr(\Sigma_2))$ is a morphism of $R$-linear
supercategories with shift.
\end{Proposition}

\

\begin{Proposition}
$\Gr$ is a functor from $\mathrm{RICat}$ to $\mathrm{RSCat}^s$. 
\end{Proposition}

\subsection{The even subcategory of an $R$-linear supercategory with shift}

A quasi-inverse of the supercompletion functor $\Gr$ can be constructed as follows, where the proof 
of the various statements is left to the reader. 

\

\begin{Proposition}
Let $(\cT,\Sigma)$ be an $R$-linear supercategory with shift. 
Then $\Sigma$ is an involution of the even subcategory $\cT^\0$. 
\end{Proposition}

\

\begin{Proposition}
Given a morphism of $R$-linear supercategories with shifts
$F:(\cT_1,\Sigma_1)\rightarrow (\cT_2,\Sigma_2)$, consider the functor
$\Ev(f):\Ev(\cT_1)=\cT_1^\0\rightarrow \Ev(\cT_2)=\cT_2^\0$ obtained by
restricting $F$ to the subcategory $\cT_1^\0$ of $\cT_1$.  Then
$\Ev(f)$ is a morphism in $\mathrm{RICat}$ from $(\cT_1^\0,\Sigma_1)$ to $(\cT_2^\0,\Sigma_2)$.
\end{Proposition}

\

\begin{Proposition}
$\Ev$ is a functor from $\mathrm{RSCat}^s$ to $\mathrm{RICat}$.
\end{Proposition}

\

\noindent Finally, one easily proves the following:

\

\begin{Theorem}
The functors $\Gr$ and $\Ev$ are mutually quasi-inverse equivalences
between $\mathrm{RICat}$ and $\mathrm{RSCat}^s$.
\end{Theorem}

\

\noindent This shows, in particular, that $R$-linear supercategories with shift
can be reconstructed from their even part, which is an $R$-linear
category with involution.

\subsection{Calabi-Yau supercategories with shift}

In this subsection we consider the case $R=\C$. Recall the following 
notion used in \cite{nlg1}:

\

\begin{Definition}
A {\em Calabi-Yau supercategory of parity $\mu\in \Z_2$} is a pair
$(\cT,\tr)$, where:
\begin{enumerate}[A.]
\itemsep 0.0em
\item $\cT$ is a $\Z_2$-graded and $\C$-linear Hom-finite category.
\item $\tr=(\tr_a)_{a\in \Ob \cT}$ is a family of $\C$-linear maps
  $\tr_a:\End_\cT(a)\rightarrow \C$ of $\Z_2$-degree $\mu$
\end{enumerate} 
such that the following conditions are satisfied:
\begin{enumerate}[1.]
\itemsep 0.0em
\item For any two objects $a,b\in \Ob\cT$, the $\C$-bilinear pairing
  $\langle \cdot , \cdot \rangle_{a,b}:\Hom_\cT(a,b)\times
  \Hom_\cT(b,a)\rightarrow \C$ defined through:
\be
\langle t_1,t_2\rangle_{a,b}\eqdef \tr_b (t_1\circ t_2)~,~~\forall t_1\in \Hom_\cT(a,b)~,~\forall t_2\in \Hom_\cT(b,a)
\ee
is non-degenerate.
\item For any two objects $a,b\in \Ob\cT$ and any $\Z_2$-homogeneous
  elements $t_1\in \Hom_\cT(a,b)$ and $t_2\in \Hom_\cT(b,a)$, we
  have:
\ben
\label{tcyc}
\langle t_1,t_2\rangle_{a,b}=(-1)^{\deg t_1\,\deg t_2}\langle t_2,t_1\rangle_{b,a}~~.
\een
\end{enumerate}
\end{Definition}

\

\noindent We are interested in the case of Calabi-Yau supercategories
which admit a shift functor compatible with the traces $\tr_a$.

\

\begin{Definition}
\label{def:compat}
A Calabi-Yau supercategory of parity $\mu\in\Z_2$ {\em with compatible
  shift functor} is a triplet $(\cT,\tr,\Sigma)$ such that:
\begin{enumerate}
\itemsep 0.0em
\item $(\cT,\tr)$ is a Calabi-Yau supercategory of parity $\mu$.
\item $\Sigma$ is a parity change functor on the $\Z_2$-graded $\C$-linear category $\cT$.
\item We have: 
\be
\tr_{\Sigma(a)}(\Sigma(t))=\tr_\Sigma(t)~,~~\forall a\in \Ob\cT~,~\forall t\in \End_\cT(a)~~.
\ee
\end{enumerate} 
\end{Definition}

\subsection{Serre functors}

\noindent Recall the notion of Serre functor introduced by Bondal
and Kapranov \cite{BK}\footnote{Notice that we do not require $\cT$ to
  be an additive category. Also notice that we do not require $\cT$ to be triangulated.}:

\

\begin{Definition} 
A {\em Serre functor} on a Hom-finite $\C$-linear category $\cC$ is a linear
autoequivalence $S$ of $\cC$ such that for any two objects $a,b$ of
$\cC$, there exists a linear isomorphism: 
\be 
\Hom_{\cC}(a, S(b))\simeq \Hom_{\cC}(b,a)^\vee
\ee 
which is natural in both $a$ and $b$. More precisely, there exists an isomorphism of functors: 
\be
\Hom_{\cC}(\id_\cC\times S)\simeq \rD\circ \Hom_{\cC}\circ \tau~~,
\ee
where $\tau:\cT\times \cT\rightarrow \cT\times \cT$ is the
transposition functor and $\rD:\vect_\C\rightarrow \vect_\C$ is the
dualization functor on the category $\vect_\C$ of finite-dimensional vector spaces. 
\end{Definition}

\

\noindent One has the following equivalent description:

\

\begin{Proposition}
\label{prop:Serre}
Let $\cC$ be a Hom-finite $\C$-linear category and let $S$ be a linear
automorphism of $\cC$. Then the following statements are equivalent:
\begin{enumerate}[(a)]
\itemsep 0.0em
\item $S$ is a Serre functor for $\cC$. 
\item  For any object $a$ of $\cC$, there exists a linear map
  $\tr_a:\Hom_{\cC}(a,S(a))\rightarrow \C$ such that the following
  conditions are satisfied for any objects $a$ and $b$ of $\cC$:
\begin{itemize}
\item $\tr_a(g\circ f)=\tr_b(S(f)\circ g)~,~\forall f\in \Hom_\cC(a,b)~,~\forall g\in \Hom_{\cC}(b,S(a))$~.
\item The bilinear map $\langle \cdot, \cdot
  \rangle_{a,b}^S:\Hom_\cC(a,b)\times \Hom_{\cC}(b,S(a))\rightarrow \C$
  defined through:
\be
\langle f,g\rangle_{a,b}^S\eqdef \tr_b(S(f)\circ g)~,~\forall f\in \Hom_\cC(a,b)~,~\forall g\in \Hom_{\cC}(b,S(a))
\ee
is non-degenerate.
\end{itemize}
\end{enumerate}
\end{Proposition}

\subsection{Calabi-Yau categories with involution}

\

\

\noindent Given a $\C$-linear category $\cT$ with involution and an element $\mu\in \Z_2$, we define: 
\be
\Sigma^\mu\eqdef\twopartdef{\id_{\cT^\0~~}}{\mu=\0 ~~ ,}{\Sigma~~}{\mu=\1 ~~ .}
\ee

\

\begin{Definition}
Let $\mu\in \Z_2$. Then a $\C$-linear category with involution
$(\cC,\Sigma)$ is called {\em $\mu$-Calabi-Yau} if $\Sigma^\mu$ is a
Serre functor for $\cC$.
\end{Definition}

\

\noindent For any $\mu\in \Z_2$, let $\mathrm{ICYCat}(\mu)$ denote the
full subcategory of $\mathrm{{\C}ICat}$ consisting of all
$\mu$-Calabi-Yau categories with involution and
$\mathrm{SCYCat^s}(\mu)$ denote the full subcategory of
$\mathrm{{\C}SCat^s}$ consisting of all Calabi-Yau supercategories of
signature $\mu$. The proof of the following statement is immediate:

\

\begin{Proposition}
\label{prop:SerreFunctor}
The restrictions of the functors $\Gr$ and $\Ev$ give mutually
quasi-inverse equivalences between the categories
$\mathrm{ICYCat}(\mu)$ and $\mathrm{SCYCat^s}$.
\end{Proposition}

\subsection{Shift functors on $\Z_2$-graded dg-categories}

\

\

\begin{Definition}
Let $R$ be a unital commutative ring and $\cA$ be $\Z_2$-graded and
$R$-linear dg-category. A {\em differential shift functor} on $\cA$ is
a shift functor $\Sigma$ on the underlying $R$-linear supercategory of
$\cA$ which satisfies the following condition for any objects $a$ and
$b$ of $\cA$:
\ben
\label{dSigma}
\Sigma_{a,b}\circ \dd_{a,b}=\dd_{\Sigma(a),\Sigma(b)}\circ \Sigma_{a,b}~: 
\Hom_{\cA}(a,b)\rightarrow \Hom_{\cA}(\Sigma(a),\Sigma(b)) ~~,
\een
where $\dd_{a,b}$ and $\dd_{\Sigma(a),\Sigma(b)}$ are the odd
differentials on the $R$-supermodules $\Hom_{\cA}(a,b)$ and 
respectively $\Hom_{\cA}(\Sigma(a),\Sigma(b))$.
\end{Definition}

\

\noindent A differential shift functor $\Sigma$ on $\cA$ induces a shift functor on
the total cohomology category $\rH(\cA)$, which we again denote by $\Sigma$.

\begin{acknowledgements}
This work was supported by the research grant IBS-R003-S1.
\end{acknowledgements}

\end{document}